\documentclass{amsart}
\usepackage{amssymb,amscd,verbatim, amsthm,amsmath,amsgen,xspace}

\usepackage[normalem]{ulem}

\usepackage{graphicx,color}

\newcommand\clrr{\color{red}}
\newcommand\clrblu{\color{blue}}

\newcommand\ovl{\overline}

\newcommand \wh[1]{\widehat{#1}}

\newcommand \fk[1]{{{\mathfrak #1}}}
\newcommand \C[1]{{\mathcal #1}}
\newcommand \ov[1]{{\overline {#1}}}
\newcommand \ch[1]{{\check{#1}}}
\newcommand \bb[1]{{\mathbb #1}}

\newcommand \wti[1]{{\widetilde {#1}}}
\newcommand \wht[1]{{\widehat {#1}}}

\newcommand \bC{{\bb C}}

\newcommand \bR{{\bb R}}

\newcommand\ie{{\it i.e.~ }}

\newcommand\eg{{\it e.g.~ }}

\newcommand\ep{{\epsilon}}
\newcommand\eps{{\epsilon}}
\newcommand\la{{\lambda}}

\newcommand\al{{\alpha}}
\newcommand\sig{{\sigma}}

\newcommand\twedge{\textstyle{\bigwedge}}

\newcommand\Ad{\operatorname{Ad}}
\newcommand\Ext{\operatorname{Ext}}
\newcommand\Hom{\operatorname{Hom}}

\newcommand\Ind{\operatorname{Ind}}

\newcommand\Pin{\operatorname{Pin}}

\newcommand\diag{\operatorname{diag}}
\newcommand\reg{\operatorname{reg}}
\newcommand\DI{\operatorname{DI}}

\renewcommand\dim{\operatorname{dim}}

\renewcommand\ch{\operatorname{ch}}

\renewcommand\det{\operatorname{det}}
\newcommand\EP{\operatorname{EP}}
\renewcommand{\ell}{\operatorname{ell}}

\newcommand\im{\operatorname{Im}}
\renewcommand\ker{\operatorname{ker}}
\newcommand\coker{\operatorname{coker}}

\newcommand\supp{\operatorname{supp}}

\newcommand\Tr{\operatorname{tr}}
\newcommand\Id{\operatorname{Id}}

\newcommand\rank{\operatorname{rank}}
\newcommand\Spin{\operatorname{Spin}}

\newcommand\Cas{\operatorname{Cas}}
\renewcommand\top{\operatorname{top}}
\newcommand{\tr}{\operatorname{tr}}

\newcommand{\Rea}{\operatorname{Re}}

\newcommand\bX{{\ov{X}}}

\newcommand{\pf}{\begin{proof}}
\newcommand{\epf}{\end{proof}}
\newcommand{\eq}{\begin{equation}}
\newcommand{\eeq}{\end{equation}}
\newcommand{\eqn}{\begin{equation*}}
\newcommand{\eeqn}{\end{equation*}}
\newcommand\bed{\begin{definition}}
\newcommand\ebed{\end{definition}}
\newcommand\bethm{\begin{theorem}}
\newcommand\ebethm{\end{theorem}}
\newcommand\bealigned{\begin{aligned}}
\newcommand\ebealigned{\end{aligned}}

\newcommand{\fra}{\mathfrak{a}}
\newcommand{\frb}{\mathfrak{b}}

\newcommand{\frg}{\mathfrak{g}}
\newcommand{\frh}{\mathfrak{h}}
\newcommand{\frk}{\mathfrak{k}}

\newcommand{\frm}{\mathfrak{m}}
\newcommand{\frn}{\mathfrak{n}}
\newcommand{\fro}{\mathfrak{o}}
\newcommand{\frp}{\mathfrak{p}}

\newcommand{\frs}{\mathfrak{s}}
\newcommand{\frt}{\mathfrak{t}}
\newcommand{\fru}{\mathfrak{u}}

\newcommand{\bbC}{\mathbb{C}}

\newcommand{\bbR}{\mathbb{R}}

\newcommand{\bbZ}{\mathbb{Z}}
\newtheorem{theorem}[equation]{Theorem}
\newtheorem{corollary}[equation]{Corollary}

\newtheorem{lemma}[equation]{Lemma}
\newtheorem{proposition}[equation]{Proposition}
\newtheorem{definition}[equation]{Definition}
\newtheorem{remark}[equation]{Remark}

\numberwithin{equation}{section}

\begin{document}
%\today

\bigskip
\title{Dirac index and twisted characters}

\author{Dan Barbasch}
      \address[D. Barbasch]{Department of Mathematics\\
               Cornell University\\Ithaca, NY 14850, U.S.A.}
        \email{barbasch@math.cornell.edu}
\thanks{Dan Barbasch was supported by  NSA grant H98230-16-1-0006}
\author{Pavle Pand\v zi\'c}
	\address[P. Pand\v zi\'c]{Department of Mathematics,
          University of Zagreb, Bijeni\v cka 30, 10000 Zagreb,
          Croatia} 
	\email{pandzic@math.hr}
\thanks{P.~Pand\v zi\'c was supported by grant no. 4176 of the
  Croatian Science Foundation and by the QuantiXLie Center of Excellence.} 
\author{Peter Trapa}
\address[P. Trapa]{Department of Mathematics, University of Utah, Salt
  Lake City, UT 84112, U.S.A.}
\email{ptrapa@math.utah.edu} 

\begin{abstract} 
Let $G$ be a real reductive Lie group with maximal compact subgroup
$K$. We generalize the usual notion of Dirac index to a twisted
version, which is nontrivial even in case $G$ and $K$ do not have
equal rank.  We compute ordinary and twisted indices of standard
modules. As applications, we study extensions of Harish-Chandra
modules and twisted characters. 
\end{abstract}
 
\maketitle

\bigskip

\section{Introduction}\label{section intro}
\subsection{}
{
The Dirac operator has played an important role in representation
theory, in particular the realization and properties of the Discrete
Series, work of Parthasarathy, Schmid, Atiyah-Schmid. On the other
hand, the use of the Dirac operator has been an important tool in  the
determination of the unitary dual and cohomology of discrete groups
via the Dirac inequality, as it appears in the work of Borel, Enright,
Kumaresan, Parthasarathy, Salamanca-Riba, Vogan, Zuckerman and many
others. This has led to the notion of Dirac cohomology, introduced by
Vogan, and developed further by the work of Huang, Pand\v zi\' c and others.

\medskip
In this paper we develop further the connections between the Dirac
cohomology of a representation and its distribution
character as well as its possible extensions as encoded in the
$\Ext-$groups. We work in a more general setting than the 
results alluded to earlier, for example  our results are for groups of unequal
rank and possibly disconnected. An important  motivating example is the case 
of  a complex group viewed as a real group.  Most of
the results are stated for the case of a real Lie group which is the
real points of a linear algebraic connected reductive Lie group, though they 
might hold for a larger class.

The next sections in the introduction review 
the known material on Dirac cohomology with particular attention to
the modifications necessary to treat the possible disconnectedness of
the group. We then give a detailed statement of the results.   
}
\subsection{}
Let $G:=G(\mathbb R)$ be the group of real points of a reductive linear
algebraic connected group $G(\bC)$. 
In particular, $G$ is in the {\it Harish-Chandra class}, \ie  it has only finitely many connected components,
the derived group $[G,G]$ has finite center, and the automorphisms $\Ad(g)$, $g\in G$, of  the complexified Lie algebra
$\frg$ of $G$ are all inner. Furthermore the Cartan subgroups are abelian.

\medskip
Let $\theta$ be  a Cartan involution of $G$ and 
$K=G^\theta$ the corresponding maximal compact subgroup. We do not assume that $K$ is connected. 
Let $\frg_0=\frk_0\oplus\frs_0$ be the Cartan decomposition of the
Lie algebra $\frg_0$ of $G$ corresponding to $\theta$. We denote the 
linear extension of $\theta$ to the complexification $\frg$ of
$\frg_0$ again by
$\theta$; let $\frg=\frk \oplus\frs$ be the corresponding Cartan decomposition
(\ie the decomposition into the $\pm 1$ eigenspaces for $\theta$). 

Fix $B$, a nondegenerate  invariant symmetric bilinear form on $\frg$ which is
negative definite on $\fk k_0$ and positive definite on $\frs_0.$
Let $C(\frs)$ be the Clifford algebra of $\frs$ with  respect to  $B$
and let $U(\frg)$ be the universal enveloping algebra of $\frg$. 

The Dirac operator $D$ is defined as
\[
D=\sum_i b_i\otimes d_i\in U(\frg)\otimes C(\frs),
\]
where $b_i$ is any basis of $\frs$ and $d_i$ is the dual basis with
respect to $B$. Then $D$ is independent of the choice of the basis
$b_i$. The square of $D$ is given by
the following formula due to  Parthasarathy \cite{P1}:
\eq
\label{Dsquared}
D^2=-(\Cas_\frg\otimes 1+\|\rho_\frg\|^2)+(\Cas_{\frk_\Delta}+\|\rho_\frk\|^2).
\eeq
Here $\Cas_\frg$ is the Casimir element of $U(\frg)$ and
$\Cas_{\frk_\Delta}$ is the Casimir element of $U(\frk_\Delta)$, 
where $\frk_\Delta$ is the diagonal copy of $\frk$ in $U(\frg)\otimes
C(\frs)$, defined using the obvious embedding 
$\frk\hookrightarrow U(\frg)$ and the usual map
$\frk\to\frs\fro(\frs)\to C(\frs)$. See \cite{HP2} for details. 

\begin{definition}
\label{def pin}
We will denote by $K^\dagger$ the {\it pin double cover} of $K$, and by
superscript $\dagger$ the various double covers of subgroups of $K$. 
(We note that it is more usual to denote the pin or spin double covers by superscript $\wti{\ }$, but in this paper we reserve this notation for a component of the extended group defined below.)

Recall that $K^\dagger$ is obtained from the pullback diagram
\[
\begin{CD}
K^\dagger @>>> \Pin(\frs_0) \\
@VVV @VVpV \\
K @>>> O(\frs_0),
\end{CD}
\]
where the bottom arrow is the action map and the right arrow is the
canonical double covering map (see e.g. \cite{HP2} for the definition of the
$\Pin$ group and the covering map).

In other words,
\[
K^\dagger=\{(k,g)\in K\times\Pin(\frs_0)\,\big|\, \Ad(k)\big|_{\frs_0}= p(g)\},
\]
where $p:\Pin(\frs_0)\to O(\frs_0)$ is the covering map.
\end{definition}

If $X$ is a $(\frg,K)-$module, and if $S$ is a spin module for $C(\frs)$,
then $X\otimes S$ is a $(U(\frg)\otimes C(\frs),K^\dagger)-$module, with the
action of $u\otimes c\in U(\frg)\otimes C(\frs)$ given by
\[
(u\otimes c)(x\otimes s)=ux\otimes cs,\qquad x\in X, s\in S,
\]
and the action of $(k,g)\in K^\dagger$ given by
\[
(k,g)(x\otimes s)=kx\otimes gs,\qquad x\in X, s\in S.
\]
(Recall that $gs$ is defined since $\Pin(\frs_0)\subset C(\frs_0)\subset C(\frs)$.)
Here we consider $(U(\frg)\otimes C(\frs),K^\dagger)$ as a pair in the usual sense,
with $(k,g)\in K^\dagger$ acting on $u\otimes c\in U(\frg)\otimes C(\frs)$ by
\[
(k,g)(u\otimes c)=\Ad(k)u\otimes gcg^{-1},
\]
and with the Lie algebra $\frk$ of $K^\dagger$ embedded into $U(\frg)\otimes C(\frs)$
as $\frk_\Delta$. (Note that our definition of the $K^\dagger-$action on $U(\frg)\otimes C(\frs)$
is different from the one used in \cite{DH}, so the problem noticed in \cite{DH}, p.41--42,
does not appear here.)

In particular, $D$ acts on $X\otimes S$ and one can define the Dirac cohomology of $X$ as 
\[
H_D(X)=\ker D / \ker D\cap\im D.
\]
$H_D(X)$ is a $K^\dagger-$module, since the elements of $K^\dagger$ which map to
the even part $\Spin(\frs_0)$ of $\Pin(\frs_0)$ commute with $D$, while
the elements of $K^\dagger$ which map to
the odd part of $\Pin(\frs_0)$ anticommute with $D$.

In the rest of the paper we assume that $X$ is admissible and has infinitesimal character. In particular,
it follows that the Dirac cohomology $H_D(X)$ is finite-dimensional.

If $X$ is unitary or finite-dimensional, then
\[
H_D(X)=\ker D=\ker D^2.
\]
Let $\frh_0=\frt_0\oplus\fra_0$ be a fundamental Cartan subalgebra of
$\frg_0$, and $\frh=\frt\oplus\fra$ its complexification. 
We view $\frt^*$ as a subspace of $\frh^*$ by extending
functionals on $\frt$ by 0 over $\fra$. We fix compatible positive
root systems $R^{+}_{\frg}$ and $R^{+}_{\frk}$ for $(\frg,\frh)$
respectively $(\frk,\frt)$. In particular, this  determines
\textit{half sums of roots} 
$\rho_\frg$ and $\rho_\frk$ as usual. Write $W_{\frg}$
(resp. $W_{\frk}$) for the Weyl group associated with the roots of 
$(\frg,\frh)$ (resp. $(\frk,\frt)$). We identify
infinitesimal characters with elements of $\frh^*$ via the
Harish-Chandra isomorphism.  

{
The following result was conjectured by Vogan \cite{V}, and proved in \cite{HP1} for connected $G$. In this paper we are primarily concerned with disconnected groups. The extension to this case is in \cite{DH}, and follows from \cite{HP1} combined with the remark that  Dirac cohomology commutes with restriction to the connected component. 
}  

\begin{theorem}
\label{HPmain}
Let $X$ be a $(\frg,K)-$module with infinitesimal character
$\Lambda\in\frh^*$. Assume that $H_D(X)$ contains the irreducible
$K^\dagger-$module $E_\gamma$ with highest weight $\gamma\in\frt^*\subset\frh^*$.  
(When $K$ and $K^\dagger$ are disconnected, we follow the conventions
of \cite{greenbook}, Section
  5.1. Essentially, $\gamma$ can be any choice of a  highest weight of
  the restriction of $E_\gamma$ 
  to the connected component $K^\dagger_0$. )

Then $\gamma+\rho_\frk=w\Lambda$ for some $w\in W_\frg$. In other
words, the $\frk-$infinitesimal character of any $K^\dagger-$type
contributing to $H_D(X)$  is conjugate to the $\frg-$infinitesimal
character of $X$ by the Weyl group $W_\frg$. 
\end{theorem}

Further work on Dirac cohomology and  unitary modules (\cite{B},
\cite{Bs}),  includes partial answers to the following Problem:

\medskip
\centerline{\textit {Determine the Dirac cohomology of any irreducible
    unitary  module.}} 

\bigskip
There are also various applications, \eg  relations to other kinds of
cohomology. For more details, see \cite{HP1}, \cite{HP2}, \cite{HPR},
\cite{HKP}, \cite{BP1} and \cite{BP2}. Similarly there is a parallel
theory for affine graded Hecke algebras, \cite{BCT}. 

\bigskip

The following notion is a generalization of the usual notion of Dirac index in the equal rank case (we recall the usual notion below). 
Before defining this generalized Dirac index, we recall the notion of virtual modules.
{
\bed 
\label{def virtual}
For an arbitrary compact group $\C K,$ let $\C C$ be the category of
admissible representations of $\C K$. In other words, each object of $\C
C$ is a direct sum of irreducible (finite-dimensional) representations
of $\C K$, with finite (nonnegative) multiplicities. 

The Grothendieck group of $\C C$ is then
\[
\C G(\C C)=\bbZ^{\widehat{\C K}}=\{\sum_{\gamma\in\widehat{\C K}}
n_\gamma\gamma\,\big|\, n_\gamma\in\bbZ\}, 
\]
with the obvious addition, and $\widehat{\C K}$ denotes the set of
irreducible representations of $K$.
We call the elements of $\C G(\C C)$ \textit{virtual $K-$modules}. 
\ebed
}
In the context of this paper, $\C K$ may refer to the maximal compact
subgroup $K\subset G,$ or one of its covers $K^\dagger$ 
or $(K^+)^\dagger$.

In other words, virtual $\C K-$modules are $\bbZ-$linear combinations of
irreducible $\C K-$modules. For each object $M$ of $\C C$, its image is
denoted by $[M]$. When no confusion arises, we will omit the brackets
and denote $[M]$ simply by $M$. 

{
The equivalence of the Grothendieck group with $\bb Z-$linear
combinations of characters is standard}. In the special case
when $\C K=\{1\}$, virtual $\C K-$modules are just virtual vector
spaces; in this case $\C C$ is the category of finite-dimensional vector
spaces, and $\C G(\C C)$ can be identified with $\bbZ$ via taking the
dimension. 

\bigskip 
Let $\gamma$ be an  automorphism of $(U(\fk g)\otimes C(\frs),K^\dagger)$. This means the following:
\begin{enumerate}
\item $\gamma$ consists of an
automorphism $\gamma^\frg$ of $U(\frg)\otimes C(\frs)$ and an automorphism $\gamma^K$ of $K^\dagger$;
\item $\gamma$ is compatible with the action of $K^\dagger$ on $U(\frg)\otimes C(\frs)$ in the sense that
\[
\gamma^\frg((k,g)(u\otimes c))=\gamma^K(k,g)\gamma^\frg(u\otimes c)
\]
for $(k,g)\in K^\dagger$ and $u\otimes c\in U(\frg)\otimes C(\frs)$;
\item the differential of $\gamma^K$ coincides with the restriction of $\gamma^\frg$ to $\frk_\Delta$.
\end{enumerate}

Let now $X\otimes S$ be a $(U\frg)\otimes C(\frs),K^\dagger)-$module as above, 
with the action denoted by $\pi$.
We assume that $X\otimes S$ has a compatible action of $\gamma$, \ie, there is an operator
$\pi(\gamma)$ on $X\otimes S$ such that  
$$
\bealigned
&\pi(\gamma)\pi(u\otimes c)\pi(\gamma)^{-1}=\pi(\gamma^\frg(u\otimes c)),\quad u\otimes c\in U(\frg)\otimes C(\frs);\\
&\pi(\gamma)\pi(k)\pi(\gamma)^{-1}=\pi(\gamma^K(k)),\quad k\in K^\dagger.
\ebealigned
$$
We assume in the following that $\gamma$ is an involution, and that
\[
\gamma(D)=-D,
\]
so $\pi(\gamma)$ and $\pi(D)$ anticommute. Then $\gamma$ preserves $H_D(X)$, and $H_D(X)$ splits into $\pm 1$ eigenspaces for $\gamma$:
\[
H_D(X)=H_D(X)^+\oplus H_D(X)^-.
\]
If we denote the fixed points of $\gamma$ in $K^\dagger$ by $K^\dagger_\gamma$, then $K^\dagger_\gamma$ preserves the above decomposition, and
we define the $\gamma-$index of $D$ on $X\otimes S$ as the function
\eq
\label{chi gamma}
\chi_\gamma^X(k)=\tr(\gamma k; H_D(X))=\tr(k; H_D(X)^+)-\tr(k;H_D(X)^-),\qquad k\in K^\dagger_\gamma.
\eeq
Equivalently, the $\gamma-$index of $D$ on $X\otimes S$ is the virtual $K^\dagger_\gamma-$module
\eq
\label{def tw index}
I_\gamma(X)= H_D(X)^+ -H_D(X)^-.
\eeq
We can also decompose all of $X\otimes S$ into the $\pm 1$ eigenspaces for $\gamma$:
\eq
\label{pmdec}
X\otimes S = (X\otimes S)^+\oplus (X\otimes S)^-,
\eeq
and this decomposition is invariant under $K^\dagger_\gamma$.
Since $D$ anticommutes with $\gamma$, it interchanges these
eigenspaces, so we get two operators,
\[
D^{\pm}: (X\otimes S)^\pm\to (X\otimes S)^\mp.
\]
Now we assume that $X\otimes S$ is admissible for $K^\dagger_\gamma$; in the examples we are interested in, this will be true.
Then we have

\begin{proposition}
\label{general index}
With notation and assumptions as above, 
\[
I_\gamma(X)=(X\otimes S)^+ - (X\otimes S)^-
\]
as $K^\dagger_\gamma-$modules. 
Furthermore, $I_\gamma(X)$ is the index of the operator $D^+$, \ie
$I_\gamma(X)=\ker D^+ - \coker D^+$, while the index of the operator
$D^-$ is -$I_\gamma(X)$.  
\end{proposition}
\pf  We can decompose $X\otimes S$ into eigenspaces of $D^2$, and this
is compatible with (\ref{pmdec}). It is clear that $D^\pm$ are
isomorphisms on any eigenspace for a nonzero eigenvalue. Hence for the
first statement we have to consider only the zero eigenspace, where
$D$ is a differential, and so the claim follows from the
Euler-Poincar\'e principle. The other statements are also easy. 
\epf

Before passing to examples, we discuss the special case when $\gamma$ is
constructed from an automorphism $\gamma_1$ of the pair $(\frg,K)$, and an
automorphism $\gamma_2$ of $C(\frs)$. As before, to be called an automorphism
of the pair $(\frg,K)$, $\gamma_1$ should consist of an automorphism $\gamma_1^\frg$
of $\frg$ and an automorphism $\gamma_1^K$ of $K$, such that the differential of
$\gamma_1^K$ equals the restriction of $\gamma_1^\frg$ to $\frk$, and such that
\[
\gamma_1^\frg(\Ad(k)Y)=\Ad(\gamma_1^K(k))\gamma_1^\frg(Y),\qquad k\in K,Y\in\frg.
\]
Now we set 
\[
\gamma^\frg=\gamma_1^\frg\otimes \gamma_2:U(\frg)\otimes C(\frs)\to U(\frg)\otimes C(\frs)
\]
and
\[
\gamma^K(k,g)=(\gamma_1^K(k),\gamma_2(g))\qquad (k,g)\in K^\dagger.
\]
It is easy to check that $\gamma^\frg$ and $\gamma^K$ define an automorphism of the pair
$(U(\frg)\otimes C(\frs),K^\dagger)$, if $\gamma^K$ is well defined, \ie 
\[
(k,g)\in K^\dagger\quad \Rightarrow\quad (\gamma_1^K(k),\gamma_2(g))\in K^\dagger.
\]
In other words, the condition on $\gamma_1$ and $\gamma_2$ is
\eq
\label{condition gamma}
\Ad(\gamma_1^K(k))\big|_{\frs_0}=p(\gamma_2(g)),\qquad \text{for all }(k,g)\in K^\dagger.
\eeq
{As before, we assume that $\gamma(D)=-D.$ }

\bigskip
We now present two examples of the above setting; they are the
main objects of study in this paper. 

\subsection{Equal Rank Case}
The first example is the ordinary Dirac index in the equal rank
case. Let  $\frh_0=\frt_0$ be the compact Cartan subalgebra in
$\frg_0$. In this case $\dim\frs$ is even, so there is only one spin
module $S$, 
and it is a graded module for $C(\frs)=C^0(\frs)\oplus C^1(\frs)$, \ie 
$S=S^+\oplus S^-$, with $S^\pm$ preserved by $C^0(\frs)$ and 
interchanged by $C^1(\frs)$. (Recall that $S$ can be constructed
as $\bigwedge\frs^+$ with $\frs^+$ a maximal isotropic subspace of $\frs$,
and that one can take $S^+=\bigwedge^{\text{even}}\frs^+$  and 
$S^-=\bigwedge^\text{odd}\frs^+$.)

Recall that $\theta$ denotes the Cartan involution of $\frg$. It induces $-\Id\in O(\frs_0),$
and so gives rise to two elements in $\Pin(\frs_0).$ It is easy to see that these elements
are
\[
\pm Z_1Z_2\dots Z_s\in C(\frs_0),
\]
where $Z_1,\dots,Z_s$ is any orthonormal basis of $\frs_0$. We fix one of these two elements, 
and call it again $\theta.$  In this way $\theta$ acts on $S$, and one easily checks that 
$S=S^+\oplus S^-$ is the decomposition into eigenspaces of $\theta$. Moreover, we can
make the choice of $\theta$ compatible with the choice of $S^\pm$, so that $\theta$ is 1 on
$S^+$ and $-1$ on $S^-$. Furthermore, we can extend the automorphism $\theta=-\Id$ of
$\frs_0$ to an automorphism of $C(\frs)$, and this automorphism is exactly the conjugation
by the element $\theta\in C(\frs)$. (This automorphism is in fact equal to the sign automorphism
of $C(\frs)$.)

We now consider the automorphism $\gamma$ of $(U(\frg)\otimes C(\frs),K^\dagger)$ constructed from
the automorphisms $\gamma_1=\Id$ of $(\frg,K)$ and $\gamma_2=\theta$ of $C(\frs)$. To see that
this makes sense, we have to check the condition (\ref{condition gamma}). This however immediately
follows from $\gamma_1^K=\Id$ and from
\[
p(\theta g\theta^{-1})=p(\theta)p(g)p(\theta^{-1})=(-\Id)p(g)(-\Id)=p(g).
\]
It is clear that $\gamma$ is an involution, and that $\gamma(D)=-D$. Moreover,
since $\theta$ is an inner automorphism of $C(\frs)$, it is clear that
$\gamma$ automatically 
acts on $X\otimes S$ for any $(\frg,K)-$module $X$. It is furthermore
clear that  
\eq
\label{K gamma ord}
K^\dagger_\gamma=\{(k,g)\in K^\dagger\,\big|\, g\in\Spin(\frs_0)\}.
\eeq
In particular, $K^\dagger_\gamma$ contains the connected component of $K^\dagger$.

We can now consider the $\gamma-$index of $D$ on $X\otimes S$, which we denote simply
by $I(X)$ in the present case. It is given by
(\ref{chi gamma}) or (\ref{def tw index}), with properties described in Proposition \ref{general index}.
In particular,
\eq
\label{index formula}
I(X)=X\otimes S^+ - X\otimes S^-.
\eeq
If $K^\dagger_\gamma=K^\dagger$, \ie, the natural map from $K^\dagger$ to $\Pin(\frs_0)$ maps $K^\dagger$ into $\Spin(\frs_0)$, 
then $I(X)$ is the usual index as in \cite{P1} and the work of
Hecht-Schmid and Atiyah-Schmid.  

\medskip
The Dirac index $I(X)$, which can be thought of as the Euler
characteristic of Dirac cohomology, is important for several
reasons. First, it is directly related to the character of $X$ on the
compact Cartan subgroup. 
This follows from (\ref{index formula}); see Section
\ref{sec:characters}, especially formula (\ref{eq:character}), for
more details. One can likewise compute characters from the Euler
characteristic of an appropriate $\frn-$cohomology, but typically
there are cancellations when using $\frn-$cohomology, and no
cancellations when using Dirac cohomology. So one can say that Dirac
cohomology is closer to the character. It is also a simpler invariant,
which is typically easier to compute than $\frn-$cohomology. 

We remark that one can replace modules with fixed infinitesimal character with
arbitrary finite length modules, if one modifies the definition of
Dirac cohomology and index as in \cite{PS}. In this paper we
only consider modules with infinitesimal character, so we do not need
this generalization.

\subsection{Unequal Rank Case}
If $\frg$ and $\frk$ do not have equal rank, then the above usual
notion of index is trivial. Namely, if $\dim\frs$ is even, 
$S^\pm$ do exist, and  
we could try to define $I(X)$ as above. It however turns out that
$S^+$ and $S^-$ are typically isomorphic as
$K^\dagger_\gamma-$modules, so $I(X)=0$ for every $X$; for example,
this is always true if $K^\dagger$ is connected. 
If $\dim\frs$ is odd, then neither of the two spin modules is graded,
so $S^\pm$ can not be defined as above. One could try to use the two
inequivalent spin modules $S_1$ and $S_2$ of $C(\frs)$, but they are
again isomorphic as $K^\dagger-$modules. So the case of real reductive
groups is different from the case of graded affine Hecke algebras
\cite{CT}, \cite{CH}, where there are two inequivalent spin modules
for $\widetilde W$ (the analogue of $K^\dagger$). 
Instead, we consider the extended group  
\[
G^+ = G \rtimes \{1,\theta\},
\] 
with $\theta$ acting on $G$ by the
Cartan involution, and with $\theta^2=1\in G$. 
{(The notation is taken from \cite{W2}.)
}

The maximal compact subgroup of  $G^+$ is
\[
K^+=K\times \{1,\theta\}.
\]
A $(\frg,K^+)-$module  $(\pi,X)$ can be thought of as a $(\frg,K)-$module with an additional action of $\theta$ by $\pi(\theta)$, which satisfies
\begin{equation}
\label{extended module}
\begin{aligned}
\pi(\theta) \pi(k)\pi(\theta) &=\pi(k),\qquad\quad k\in K;\\
\pi(\theta) \pi(\xi)\pi(\theta) &=\pi(\theta(\xi)),\qquad \xi\in \frg.
\end{aligned}
\end{equation}

We now consider the automorphism $\gamma$ of $(U(\frg)\otimes C(\frs),K^\dagger)$ built from the automorphisms $\gamma_1=\theta$ of $(\frg,K)$ and
$\gamma_2=\Id$ of $C(\frs)$. 
{
Here $K^\dagger$ still denotes the $\Pin$ double cover of $K$, not of $K^+$. 
}
The compatibility condition (\ref{condition gamma}) is now trivial, and so is the fact that $\gamma$ is an involution
satisfying $\gamma(D)=-D$. It is also clear that in this case $K^\dagger_\gamma=K^\dagger$.
Moreover, $\gamma$ acts on $X\otimes S$ whenever $X$ is a $(\frg,K^+)-$module.   

We can now consider the $\gamma-$index of $D$ on $X\otimes S$, which we denote 
by $I_\theta(X)$ in the present case, and call the twisted Dirac index of $X$. It is again given by
(\ref{chi gamma}) or (\ref{def tw index}), with properties described in Proposition \ref{general index}.

In particular, we have the following equality of virtual $K^\dagger-$modules

\eq
\label{twisted index formula}
I_\theta(X)= X^+\otimes S - X^-\otimes S, 
\eeq
where $X^\pm$ denote the $\pm 1$ eigenspaces of $\theta$ on $X$.

\medskip
This setting makes sense in the equal rank case as well. Since
$\theta=\Ad k_0$ is inner, any $(\fk g, K)-$module extends naturally
to an $G^+=G\rtimes \{1,\theta\}-$module via
$\pi(\theta)=\pi(k_0).$  The resulting
twisted index is not substantially different from the usual notion of index. 
Namely let $\tilde k_0=(k_0,\theta)\in  K^\dagger$, where
$\theta\in\Spin(\frs_0)$ is the top degree element acting by $\pm 1$
on $S^\pm$. Then $\tilde k_0$ is in $K^\dagger_\gamma$ of \eqref{K  gamma ord}.  

Let $\chi_1$ (respectively $\chi_2$) be the function defined by
\eqref{chi gamma} for the ordinary (respectively twisted) Dirac
index. These functions are both defined for any $k\in
K^\dagger_\gamma$. Since $\tilde k_0^2$ acts as the identity on $S$, we have 
\[
\begin{aligned}
\chi_1(\tilde k_0k)=&\tr(\tilde k_0k;X)\tr(\tilde k_0 k\theta;S)=\tr(\tilde k_0k;X)\tr(\tilde k_0^2 k;S)=\\
=&\tr(\tilde k_0k;X)\tr(k;S)=\chi_2(k).
  \end{aligned}
\]
So we see that the twisted Dirac index $\chi_2$ is the same as the ordinary Dirac index $\chi_1$ with the argument translated by $\tilde k_0$.

\subsection{Summary}
In Section \ref{sec:characters} we explain the relationship of the ordinary Dirac index to the character on the compact Cartan subgroup, and also
the relationship of the twisted Dirac index to the twisted character on the fundamental {\it twisted Cartan subgroup} (Definition \ref{d:twistedcsg}; see also the paragraph below). The main result is Theorem \ref{T:kchar}, which implies that the formulas (\ref{index formula}) and (\ref{twisted index formula}) can
be interpreted as formulas for the character (respectively twisted character) of $X$.

{The material in Section \ref{tw conj cl} has substantial
  overlap with the work of
  Arthur, Kottwitz, Labesse, Langlands, Moeglin, 
Shelstad and others on the
twisted trace formula. In particular the work of  Waldspurger in \cite{W1} gives explicit results for various individual cases. We give a
treatment of what we need, emphasizing the point of view of group cohomology.}

These results are needed for the determination of twisted
characters. Analogous to the regular case, twisted characters are
determined by their values on the strongly semisimple  regular set,
and being invariant under conjugation by $G,$ determined by the values
on the twisted Cartan subgroups. 

For the equal rank groups (and $\theta$) twisted Cartan subgroups
coincide with the usual Cartan subgroups. In the unequal rank case the 
main result  is Theorem \ref{tw conj cx}, which says that for
complex groups viewed as real groups, the conjugacy classes of twisted
regular semisimple elements are in one to one correspondence with
conjugacy classes of involutions in the Weyl group. The analogous
result for unequal rank real groups is of the same nature but more
complicated because there are several conjugacy classes of Cartan
subgroups. We discuss examples in Subsection \ref{s:gln}. 

In Section \ref{sec:ker} we obtain certain integral formulas for
  characters of real induced modules. Such results  are well known
  under somewhat more restrictive hypotheses. Our exposition was
  influenced by the notes of Paul  Garrett, available at 
  www.math.umn.edu/$\sim$garrett/m/v/characters\_ps.pdf.  The main result
  is the character formula  (\ref{eq:tchar2}).  Our main purpose is to
  obtain vanishing of characters on certain conjugacy classes of
  twisted Cartan subgroups, which in turn implies vanishing of the
  Dirac index. 

Section \ref{sec:indices_std} gives a complete discussion of standard modules and their Dirac indices (ordinary and twisted). The main results are
Theorem \ref{t:twcx} which computes twisted indices of standard
modules of complex groups, Theorem \ref{index_std_real} which computes
ordinary indices of standard modules in the equal rank case, and
Theorem \ref{tw_index_std_real} which computes twisted indices for
standard modules of real groups. Note that Theorem \ref{t:twcx} is a
special case of Theorem \ref{tw_index_std_real}, but we treat it
separately because it is easier than the general case. 

In principle, knowing the indices of standard modules, one can compute
indices of all finite length modules (with a fixed infinitesimal
character), by expressing them as $\bbZ-$linear combinations of
standard modules via the Kazhdan-Lusztig algorithm. Conversely, if we
know the index, by computing the Dirac cohomology explicitly, then we
can use the above results to get some of the Kazhdan-Luszitg
coefficients. Examples of this are given in Subsection
\ref{sec:examples}.

Finally, in Section \ref{s:ext}, we study the relationship between
indices (ordinary or twisted) and extensions of
$(\frg,K)-$modules. The main notion we consider is the {\it
  Euler-Poincar\'e pairing} defined as the alternating sum of $\Ext$
groups between two $(\frg,K)-$modules (see (\ref{EP}) below).

\medskip
{Let $\C K$ be a compact group and }{ let $(\mu,X)$ and $(\eta, Y)$ be $\C K-$modules. Then $\Hom_\bb C[X,Y]$ inherits a
  $\C K-$module structure as well:
$$
(k\cdot\la)(x):= \eta({k})\circ \la\circ \mu({k^{-1}}).
$$
If $\C H\subset \C K$ is a normal subgroup, $\Hom_{\C H}[X,Y]$
inherits a structure of $\C K-$module, which in fact drops down to a
$\C K/\C H-$structure. This extends to the Grothendieck group in the
standard way. Precisely if $X=\sum n_iV_i$ and $Y=\sum m_jV_j$ one
applies the definition to $\sum n_im_j\Hom_{\C H}[V_i,V_j].$ 

\medskip
In particular, this applies to $\Ext_{(\fk g, K)}(X,Y)$ which is
the $i-$th cohomology of the complex
\[
\Hom_K(\twedge^\star \frs\otimes X,Y)
\]
with the usual de Rham type differential. 
For any two finite length $(\frg,K)-$modules $X$ and $Y$, the vector spaces $\Ext^i_{(\frg,K)}(X,Y)$ are finite dimensional; see \eg   
\cite{BW}, Chapter 1. Furthermore, $\Ext^i_{(\frg,K)}(X,Y)$ is zero unless $0\leq i\leq s$ where $s=\dim\frs$. Also, if $X$ and $Y$ have infinitesimal character, then all
$\Ext^i_{(\frg,K)}(X,Y)$ vanish unless the infinitesimal characters of  
$X$ and $Y$ are the same.

}

The usual Euler-Poincar\'e pairing is defined on $X$ and $Y$ as the virtual
vector space 
\eq
\label{EP}
\EP(X,Y)=\sum_{i=0}^s (-1)^i\Ext^i_{(\frg,K)}(X,Y).
\eeq
(Since the Grothendieck group of finite-dimensional vector spaces is
isomorphic to $\bbZ$ via taking dimensions, one can also think of
$\EP(X,Y)$ as being an integer.) It is easy to see that $\EP$ is
additive with respect to short exact sequences in each variable, so
it makes sense on the level of the Grothendieck
group of finite length $(\frg,K)-$modules, \ie  $X$ and $Y$ above can
also be virtual $(\frg,K)-$modules. 

\medskip
Assume first that $\rank\frg=\rank\frk$. 
If $X$ and $Y$ are finite-dimensional, the Euler-Poincar\'e principle
implies that $EP(X,Y)$ equals
\begin{equation}
\label{basiceuler}
\begin{aligned}
&\sum_{i} (-1)^i\Ext_{(\frg,K)}^i(X,Y)=\sum_{i}
(-1)^i\Hom_K\left(\textstyle{\bigwedge}^i\fk s\otimes X,Y\right)=\\
&\Hom_K\left(\sum_{i} (-1)^i\textstyle{\bigwedge}^i\fk s\otimes X,Y\right)=
\Hom_K\left((S^+-S^-)\otimes(S^+-S^-)^*\otimes X, Y\right)=\\
&\Hom_{K^\dagger}\left(X\otimes (S^+-S^-), Y\otimes(S^+-S^-)\right)=
\Hom_{K^\dagger}\left(I(X), I(Y)\right).
\end{aligned}
\end{equation}
In the above computation, we have used the fact $\twedge\frs=S\otimes
S^*$, which implies 
$$
\sum_i(-1)^i\twedge^i\fk s=(S^+-S^-)\otimes
(S^+-S^-)^*,
$$
 and also (\ref{index formula}) for Dirac indices of $X$ and $Y.$ 

{
Note that $S\otimes X$ and $S^*\otimes X$ only admit an action
  of $K^\dagger,$ but the action on  $S\otimes S^*\otimes X$ factors
  to $K.$
}
 
For general $X$ and $Y$ the above computation does not make sense, as
one can not take an alternating sum of infinite-dimensional vector
spaces. However, the end result still holds, as asserted by Theorem
\ref{thm_altsum}. The proof uses the fact that standard modules
generate  the Grothendieck group, so it is, at least in principle, enough to
understand $\EP(X,Y)$ in the case when $X$ is a standard 
module $A_\frb(\lambda)$ (see Section \ref{std_real}). This special
case can be handled by a spectral sequence described in Proposition
\ref{zuckspseq}. This enables us to pass to finite-dimensional modules
where we can use (\ref{basiceuler}). This and some additional
computations lead to a proof of  the result. 

\medskip
An analogous result (still in the equal rank case) 
has recently been proved by Huang, Mili\v ci\'c and Sun
\cite{HMS} independently. They prove that the Euler-Poincar\'e pairing
of two finite 
length modules is the same as the elliptic pairing defined in \eqref{def
  ell}. Their result
combined with the equality between the elliptic pairing and the pairing of Dirac indices (\cite{H},  \cite{R}), can be used to derive our
Theorem \ref{thm_altsum}. 

\bigskip
We now drop the equal rank assumption. Let $X$ and $Y$ be modules for
the extended group $G^+=G\rtimes \{1,\theta\}$, where $\theta$ is the
Cartan involution of $G$. This means we consider $X$ and $Y$ as
$(\frg,K)-$modules with a compatible  action of $\theta$, \ie  of the
group $\{1,\theta\}\cong\bbZ_2$, such that (\ref{extended module})
holds. Furthermore $\theta$  acts on 
$\bigwedge^i\frs$, by the scalar $(-1)^i$. Thus  the
complex  
\eq
\label{extcx}
\Hom_K(\twedge\frs\otimes X,Y)\cong\Hom_K(\twedge\frs,\Hom_\bbC(X,Y))
\eeq
is  a virtual module for $K^+/K\cong \{1,\theta\}$, with $\theta$
acting simultaneously on 
$\bigwedge\frs$, $X$ and $Y$. It is easy to check that this action of
$\theta$ commutes with the differential of the complex, so $\theta$
also acts on the cohomology of the complex, \ie  on each
$\Ext^i_{(\frg,K)}(X,Y)$. We can 
now consider
\[
\EP_\theta(X,Y)=\sum_{i=0}^s (-1)^i\Ext^i_{(\frg,K)}(X,Y),
\]
not as a virtual vector space, but as a virtual $\{1,\theta\}-$module.
We want to study the trace of $\theta$ on $\EP_\theta(X,Y)$, in the following sense.

{
\begin{definition}\label{def:c}
Let $\C K$ be a compact group and let $\C H$ be a normal subgroup of $\C K$. 
Let  $V=\sum m_j V_j$ be a finite-dimensional (virtual)
$\C K-$module, with $\C H$ acting trivially.
Then the trace of $k\in\C K/\C H$ on $V$ is the usual
\[
\Tr(k,V)=\sum m_j\Tr(k,V_j).
\]
Since $\C K$ is compact, $k$ acts semisimply. For each
irreducible module,
$$
\Tr(k,V_j)=\sum \tau \dim V_\tau,\qquad V_\tau=\{ v\in V_j\ :\ k\cdot
v=\tau v\}.
$$ 
We will often identify $\Tr (k,V_j)$ with the virtual vector
space 
$$
\Tr(k,V_j)=\sum \tau V_\tau.
$$
 \end{definition}
This definition will be used in the context of $\C K=K^+$ or $(K^+)^\dagger$ and $\C H=K$ or $K^\dagger.$ The element $k$ will simply
be $\theta$ or $k\theta$ with $k\in K.$ 
}

\bigskip
If $X$ and $Y$ are finite-dimensional, we can write the following
equalities of virtual vector spaces:
\begin{equation}
  \label{eq:twistedext}
  \begin{aligned}
&\Tr\big[\theta,\EP_\theta(X,Y)\big]=\\
&\Tr\big[\theta\otimes\theta,
\Hom_K\big[\sum (-1)^i\twedge^i \fk s,\Hom(X,Y)\big]\big]=\\ 
&\Tr\big[1\otimes\theta,\Hom_K\big[\sum\twedge^i\fk s,\Hom(X,Y)\big]\big]=\\
&c\Tr\big[1\otimes\theta,\Hom_K\big[S\otimes S^*,\Hom(X,Y)\big]\big]=\\
&c\Hom_{K^\dagger}\big[ (X^+-X^-)\otimes S,(Y^+ -Y^-)\otimes S\big] =\\
& c  \Hom_{K^\dagger}\big[ I_\theta(X),I_\theta(Y)\big].
  \end{aligned}
\end{equation}
{
For this calculation recall that $\Hom$ for virtual $K-$modules is defined by
\[
\Hom_K(\sum\lambda_iV_i,\sum\mu_jV_j)=\sum\lambda_i\mu_j\Hom_K(V_i,V_j),
\]
and likewise for $K^\dagger$.)
}

The last equality in \eqref{eq:twistedext} uses the formula (\ref{twisted index formula}), which
gives an equality of virtual $K^\dagger-$modules, without the $\theta-$action.
The constant $c$ is as follows. If $\dim \fk s$ is even, $c=1$; in this
case $\twedge\fk s = S\otimes S^*$, where $S$ is the unique 
spin module for $C(\frs)$. If $\dim\fk s$ is odd, $c=2$. There are two spin
modules $S_1$ and $S_2$, and $\twedge\fk s = S_1\otimes S_1^*\oplus
S_2\otimes S_2^*$. Since $S_1$ and $S_2$ are isomorphic as
$K^\dagger-$modules, we denote either one of them by $S$ and write
$\twedge\fk s = 2\, S\otimes S^*$.

\begin{remark}
\label{rmk twist}
{\rm If we do not bring the $\theta-$action into play, and consider
  just the virtual vector space $\EP(X,Y)$, then in the case when $\frg$
and $\frk$ do not have equal rank, $\EP(X,Y)=0$ for all
$X$ and $Y$; it is enough to check this for $X$ a standard module,
and for that case, one can use Lemma \ref{htext}. 
With the $\theta-$action taken into account, this does not follow. Namely, 
the conclusion of Lemma \ref{htext} is no longer valid, since  
$\sum (-1)^i\twedge^i \fk a$ is nonzero as a virtual $\{1,\theta\}-$module.
}
\end{remark}

As before, the computation  (\ref{eq:twistedext}) does not make sense for infinite-dimensional $X$ and $Y$, but the end result still holds.
This is the content of Theorem \ref{thm_altsum_theta}, which is an analogue of Theorem \ref{thm_altsum} in the twisted setting, and it has a similar proof.

{
As in the untwisted case, we can connect the two pairings, $\EP_\theta(X,Y)$ and $\Hom_{K^\dagger}\big[ I_\theta(X),I_\theta(Y)\big]$, to a third kind of pairing, called the twisted elliptic pairing of $X$ and $Y$. We discuss this at the end of Section \ref{s:ext}.
}

Finally, we stress that Theorem \ref{thm_altsum} and Theorem \ref{thm_altsum_theta} both give equalities of virtual vector spaces, and not of virtual
$\{1,\theta\}-$modules. Namely, in all Euler characteristic arguments the crucial point is to get cancellations, and 
{
these are only possible in the setting of virtual vector spaces, or virtual $K^\dagger-$modules, 
and not in the setting of virtual $\{1,\theta\}-$modules. Here is a typical example: let $V$ be a finite-dimensional vector space, and let $V^\pm$ denote the $\{1,\theta\}-$module equal to $V$ as a vector space, with $\theta$ acting by $\pm 1$. Then $V^+-V^-$ is zero as a virtual vector space, but it is nonzero as a virtual $\{1,\theta\}-$module. This is analogous to Remark \ref{rmk twist}.
}

A strong vanishing result for two discrete series $X$ and $Y$ follows from  Schmid's formula for the $\bar\fru-$cohomology:
\begin{equation}
  \label{eq:schmid}
\bigoplus_i \Ext^i_{(\frg,K)}(X,Y) = \Hom_{K^\dagger}(I(X),I(Y))=\begin{cases}
  0 &\text{ if } X\neq Y ;\\
\bbC &\text{ if } X=Y,
\end{cases}
\end{equation}
The Dirac index at the end of \ref{eq:schmid}) is always a single $K^\dagger-$type, and different $X$ and $Y$ have different indices. Conceivably the vanishing of $\Ext$ could be established using a calculation similar to (\ref{eq:twistedext}). We do not know of such an argument.

\section{$K-$characters as distributions}
\label{sec:characters}
The results in this section go back to one of Harish-Chandra's early
papers \cite{HC}. We give some details since we consider groups which
are not in the Harish-Chandra class. 
\subsection{General manifolds} Let $(M,\varpi_M)$ and
$(N,\varpi_N)$ be manifolds with orientation forms. Assume
$\Psi:M\longrightarrow N$ is a submersion. 
\begin{lemma}
For every $n\in N,$ $\Psi^{-1}(n)$ is a submanifold. There is a form
$\eta_n$ on $\Psi^{-1}(n)$ such that
\[
\varpi_M(m)=(\Psi^*\varpi_N)(m)\wedge\eta_{\Psi(m)}(m).
\]
\end{lemma}
\begin{corollary}
There is an onto map $\wti\Psi:C_c^\infty(M)\longrightarrow
C_c^\infty(N)$, denoted\newline $\wti\Psi(\phi):=F_\phi$,  given by the
formula
\[
F_\phi(n)=\int_{\Psi^{-1}(n)}\phi(x)\eta_n(x).
\]
\end{corollary}

\begin{definition}
  If $\Theta$ is a distribution on $N$ let 
\[
\Psi^*(\Theta)(\phi)=\Theta(F_\phi).
\]
\end{definition}
\subsection{Special Case} Let $G^+$ be a real reductive group with
maximal compact subgroup $K^+$, the fixed points of a Cartan involution
$\theta\in G^+.$ We do not assume the group is connected, for example
$G^+$ could be $G\rtimes\{1,\theta\}$ as described earlier. So $\theta$ may
not be in the connected component of the identity. Let $G^+_{\reg}$ be the
regular set, $(K^+)':=G^+_{\reg}\cap K^+.$ Write $S:=G^+/K^+$, so that $G/K\cong S=\exp\fk s$ where $\fk g=\fk k \oplus\fk s$ is the
Cartan decomposition. 

We choose $M:=(S\times (K^+)',dsdk)$, and 
$N:=((G^+)_{\reg}^{ell},dg)$. Let $\Psi(s,k):=sks^{-1}.$ Then
$(G^+)^{ell}_{\reg},$ the set of regular elliptic elements, is the
image of $\Psi$.
 The results in the previous section imply  
$\eta(s,k)=\Delta(s)ds$ where $\Delta(s)$ is the Jacobian.

Let $\Theta:=\Theta_\pi$ be the distribution character of an
admissible $(\fk g,K^+)-$module. Bouaziz \cite{Bz} has extended
Harish-Chandra's results on characters to  a larger class of groups so
that there exists a function $\Theta_\pi$ analytic on $G^+_{\reg}$
so that  
\[
\Theta(f)=\tr\pi(f)=\int_G\Theta_\pi(x)f(x)\; dx.
\]
Let $\phi=g(s)f(k)$ for a $g\in C_c^\infty(S)$ and $f\in
C_c^\infty((K^+)').$ Then 
\begin{equation}
  \label{eq:lift}
  \Theta(F_\phi)=\int_S g(s)\Delta(s)\;ds\int_{K^+}\Theta_\pi (k)f(k)\;dk.
\end{equation}
Now let $g$ depend on a parameter $t\in \bR$ such that $\supp g_t\to
\{1\}$ as $t\to 0$ and $\int_S g_t(s)\Delta(s)\;ds=1.$  We conclude
\begin{equation}
  \label{eq:limit}
  \lim_{t\to 0}\Theta(F_{\phi_t})=\int_K\Theta_\pi(k)f(k)\;dk.
\end{equation}
\begin{proposition}
 If $f\in C_c^\infty(K'),$ then $\pi(f):=\int_K f(k)\pi(k)\;dk$ is
  trace class.
\end{proposition}
\begin{proof}
The proof is the same as in \cite{Kn}.
\end{proof}
\begin{theorem}
\label{T:kchar}
  The distribution $\tr\pi(f)$ for $f\in C_c((K^+)')$ equals
  $\int_{K^+}\theta_\pi(f)f(k)\;dk.$

\end{theorem}
\begin{proof}
\[
\pi(F_{\phi_t})=\int_S\int_{K^+} g_t(s)\Delta(s)f(k)\pi(s^{-1}ks)\;ds\;dk=
\int_Sg_t(s)\Delta(s)\pi(s^{-1})\pi(f)\pi(s)\;ds.
\]
The operators $\pi(s^{-1})\pi(f)\pi(s)$ are all trace class, and
$\tr[\pi(s^{-1})\pi(f)\pi(s)]=\tr\pi(f).$  Since
$\supp g_t\to \{1\},$ we conclude
\begin{equation}
\label{eq:char}
\lim_{t\to 0}\int_S g_t(s)\Delta(s)\tr[\pi(s^{-1})\pi(f)\pi(s)]=\tr\pi(f).  
\end{equation}
The  formula follows by comparing (\ref{eq:limit}) with (\ref{eq:char}). 
\end{proof}
In the equal rank situation with $\gamma=1\otimes\theta,$ (\ref{index formula})
implies  
\begin{equation}\label{eq:character}
\Theta(X)\mid_{T_{\reg}}=\frac{\ch(I(X))}{\ch(S^+-S^-)}\bigg|_{T^\dagger_{\reg}}.
\end{equation}
{
Here $\ch(V)$ denotes the usual character of a virtual finite-dimensional $K^\dagger-$module $V$.
}
Since $\ch(S^+-S^-)\mid_{T^\dagger}$ is the noncompact part of the Weyl denominator
for $\frg$, it does not vanish identically. 

In the general case with $\gamma=\theta\otimes 1,$ we use the group
$G^+=G\rtimes\{1,\theta\}$. 
Let $T\subset K$ be a Cartan subgroup of $K.$ 
Then $\theta T$ contains elements in $(K^+)';$ in
fact the elements in $(K^+)'$ in the same connected component as $\theta$ are
all conjugate to $\theta T_{\reg}.$ Then 
{
\begin{equation}
  \label{eq:twcharacter}
 \Theta(X)(\theta\, p(t))
   =\frac{\ch(I_\theta(X))(t)}{\ch(S)(t)},\qquad t\in T^\dagger_{\reg},
\end{equation}
where $p:T^\dagger\to T$ is the covering map.
As before, $\ch(V)$ denotes the usual character of a virtual
finite-dimensional $K^\dagger-$module $V$. 
}
Note that in this case $\theta$ acts by the identity on $S,$ so
the formula is analogous to (\ref{eq:character}).

\section{Twisted Conjugacy Classes}
\label{tw conj cl}
{
We review some well known facts about 
twisted strongly regular semisimple conjugacy classes and twisted Cartan subgroups; see in particular \cite{W1}. 
}
\subsection{}
Recall  $G=\bb G(\bR)$ the rational points of a linear algebraic 
reductive connected group, 
$\theta$ the Cartan involution, $G^+=G\rtimes \{1,\theta\}$ and
$\wti G=G\theta$ as before. $G(\bb R)$ is the fixed points of the
conjugation $\sigma.$

Given an element $x\in G,$ its \textbf{twisted conjugacy class} is the
set 
$$
\{gx\theta g^{-1}\,\big|\,g\in G\}\subset \wti G
$$ 
or equivalently $\{gx\theta(g^{-1})\,\big|\,g\in G\}\subset G$.  

\begin{definition}
\label{def: str reg}
An element $g\in G^+$ is called \textbf{strongly regular} if
\newline $C_\fk g(\Ad g):=\{ X\in \fk g : \Ad g (X) =X\}$ has minimal dimension.
\end{definition}
 
\subsection{Complex Groups}
We first consider the case of  twisted conjugacy classes of strongly regular
semisimple elements for a complex group. Write $\bb G$ for $G(\bb \bC)$. 

\begin{proposition}\label{p:cxtwisted}
For any semisimple strongly regular element $\tilde x=\theta x\in \wti{\bb G}$ there
is a pair $\bb (B,\bb H=\bb T\bb A)$ stabilized by $\tilde x$ and such that $\bb T$ is the
centralizer of $\tilde x$ in $\bb G,$ and $\bb H$ is the centralizer of  $\bb T$ in
$\bb G$. 
\end{proposition}
\begin{proof}
Theorem 7.5 in \cite{St} states that any semisimple automorphism of an
algebraic group fixes a pair $(\bb B,\bb H)$ where $\bb B$ is a Borel subgroup and
$\bb H\subset \bb B$ a Cartan subgroup. Thus there is a pair
$(\bb B,\bb H)$ fixed by $\theta$.  Since any two
pairs $(\bb B,\bb H)$ and $(\bb B',\bb H')$ are conjugate by an element in $\bb G$, there
is $g\in \bb G$ such that  
$\Ad \tilde x(\bb B,\bb H)=\Ad g(\bb B,\bb H)$. By Lemma 7.3 in  \cite{St}, there is $y\in\bb G$ such
that $g^{-1}=y(\theta x)y^{-1}(\theta x)^{-1}.$ Thus there is a
$\bb G-$conjugate of $\theta x$ which preserves $(\bb B,\bb H).$ Write this
conjugate of $\theta x$ as $\theta h$ with $h\in\bb  H$; this is possible because
the (twisted) conjugate of $x$ must stabilize $(\bb B,\bb H)$ so must be in
$\bb H$. 
Write $\bb H=\bb T\bb A$ where $\theta$ acts by $1$
on $\bb T$ and $-1$ on $\bb A$. Sine $\bb G$ is complex,  any element $a\in \bb A$ can be
decomposed as $a=\theta(b)b^{-1}$, so $\theta h$ is twisted conjugate
to an element  $\theta t$ with $t\in\bb  T$.  The fact that $\bb T$ is
  the centralizer of $\tilde x$ follows from the assumption of strong regularity.
\end{proof}
\subsection{Real Groups} We now consider the case of a real group.
Recall $\bb G=G(\bC)$ the complex group, and 
$G=G(\bR)$  the fixed points of the conjugation  $\sig$.
Now assume  that $\tilde x=\theta x$ is strongly regular and fixed by $\sigma,$ \ie
$\tilde x\in \theta G(\bb R).$ By Proposition \ref{p:cxtwisted} there is a
pair $(\bb B,\bb H)$ which is $\tilde x-$stable. Write $\bb H=\bb T\bb A$
where $\bb T=C_\bb G(\tilde x)$ and $\bb H=C_\bb G(\bb T).$ Since $\sig$
stabilizes $\tilde x$, it stabilizes $\bb T$, 
and therefore also $\bb H$. Since $\bb H$ is $\sig-$stable, we can
conjugate $\tilde x$  by $G(\bR)$ so that $\bb H$ is also
$\theta-$stable. On the other hand, there is  $(\bb B_0,\bb H_0)$ so that $\bb H_0$ is
$\theta-$stable and $\sig-$stable;  $\bb H_0$ is a fundamental Cartan
subgroup with Cartan decomposition $\bb H_0=\bb T_0\bb A_0.$ In
particular $\theta \bb T_0$ has strongly regular elements, and $C_\bb
G(\bb T_0)=\bb H_0.$  By (the proof of) 
Proposition \ref{p:cxtwisted}, there is $g\in \bb G$ such that
$g\tilde xg^{-1}=\theta t$ with $t\in \bb T_0.$ Then 
$$
g\bb Tg^{-1}=C_\bb G(g\tilde xg^{-1})=C_\bb G(\theta t)\supset \bb T_0.
$$  
Since $\tilde x$ was assumed strongly regular, the dimension of the
centralizer is minimal, so in fact  $g\bb Tg^{-1}=\bb T_0.$   
Since $\bb T$ and $\bb T_0$ are both $\sig-$stable, it follows that $g\sig(g^{-1})$ normalizes $\bb T_0,$ and therefore also $\bb H_0.$ Thus it is an element of $W(\bb G,\bb H_0)$ which stabilizes $\bb T_0$. Conversely, given a $\bb T=g^{-1}\bb T_0g$ which is $\sig-$stable, the associated set of strongly regular elements is in $\theta x\bb T$ with $x=\theta(g^{-1})g.$
We conclude that  in order to classifiy strongly regular elements in
$G(\bR)\theta,$ we need to classify the $\sig-$stable tori $g\bb
T_0g^{-1}$ up to conjugacy by $G(\bb R).$ This is a Galois cohomology problem. 
If $\bb T_1$ and $\bb T_2$ are conjugate under $G(\bbR),$ then their
centralizers $C_G(\bb T_1)$ and $C_G(\bb T_2)$ are conjugate by $G(\bb R)$ as well. 
We will separate  the $g\bb Tg^{-1}$  into classes   that  have the same $\sig$ and $\theta-$stable centralizer $\bb H$.  

\begin{definition}\label{d:twistedcsg}
 Representatives of $G(\bb R)-$conjugacy classes of $\sig-$stable $\bb
 T\theta$ are called twisted   Cartan subgroups.
\end{definition}
For the real points $G\theta=G(\bR)\theta,$ the twisted Cartan
subgroups will be the fixed points under $\sig,$  denoted  $T\theta$ with $\theta$ on the right instead of on the left. 

\medskip
Recall $G:=G(\mathbb R)$, the real points of a reductive linear
algebraic group, the fixed points of a conjugation $\sigma.$  Let
$\mathfrak g$ be its Lie algebra, $\theta$ the Cartan involution,  
and $\mathfrak g=\mathfrak k \oplus\mathfrak s$ be the Cartan
decomposition. Recall $\bb H_0=\bb T_0\bb A_0$  the fundamental
$\theta-$stable Cartan subgroup with 
Cartan subalgebra $\mathfrak h_0=\mathfrak t_0 \oplus\mathfrak a_0.$  When
$G(\bbR)$ is equal rank, $\bb T_0=\bb H_0$ and twisted conjugacy classes
are just regular conjugacy classes of Cartan subgroups. So we will
concentrate  on unequal rank groups. 

We are looking for $G-$conjugacy classes of $z\bb T_0 z^{-1}$ which are
$\sig-$stable. Every $z\bb T_0 z^{-1}$ which is $\sig-$stable must have
a centralizer which is a Cartan subgroup which must  also be
$\sig-$stable.  We sort the 
$z\bb T z^{-1}$ by the  $\sig-$stable and $\theta-$stable Cartan
subgroups which are representatives of $G-$conjugacy classes of real Cartan subgroups.  
They are of the form $c\bb Hc^{-1}$ with $c$ a Cayley transform. 
A Cayley transform is a product of Cayley transforms $c_\al$ attached to noncompact
imaginary roots $\al.$ If $X_{\pm\al}$ are root vectors such that
$\theta(X_{\pm\al})=-X_{\pm\al},$ and $\sig(X_{\pm\al})=X_{\mp\al},$
then the Cayley transform is $c_\al=e^{\pi(X_\al +X_{-\al})/4}.$  Then $c_\al^2$ represents the
Weyl involution $w_\al.$ We conclude that 
$\bb T_c:=c\bb Tc^{-1}\subset \bb H_c:= c\bb H_0c^{-1}$ is $\sig$ and 
$\theta-$stable. $\bb T_c$ can therefore play the role of $\bb T_0,$ and so any $\bb T'$ with centralizer $\bb H_c$ is then conjugate to $\bb T_c$ by an   element representing a Weyl group
element in $W_c:=W(\bb G,\bb H_c).$

\medskip
We give two examples of conjugacy classes of twisted Cartan subgroups in the unequal rank case.

\subsection{Complex Groups as Real Groups}\label{s:cx}

We specialize to the case  $G^+=G\rtimes \{1,\theta\}$ with $G$ a complex group viewed as a real group,
  and consider the problem of classifying twisted regular conjugacy classes.
We start with the case of a complex group $G$ viewed as a real group.
Its Lie algebra is denoted $\fk g.$ Let $\theta$ be a Cartan
involution, with decomposition $G=KS$ and $\fk g=\fk k \oplus\fk s.$ Let
$\ovl{\phantom{a} }$ be the complex conjugation corresponding to $K;$ 
$\ovl{(ks)}=ks^{-1}=\theta(ks)$, so $\ovl{\phantom{a} }$ equals
$\theta$ in this case.   
Let $B=HN$ be a Borel subgroup with $\theta-$stable Cartan subgroup $H=TA.$ Let
$\fk h=\fk t \oplus\fk a$ be the Cartan decomposition. Then $\theta (N)=\ovl
N$ the opposite unipotent radical. 

\medskip
We complexify $G$. Recall that $\ovl{\phantom{a}}$ is conjugation in $G$
with respect to the compact form $K.$ Then
\begin{align*}
  &\bb G\cong G\times G, &&\theta_\bC(g_1,g_2)=(g_2,g_1) &&G\cong\{
  (g,\ovl{g})\ :\ g\in G\},\\
  &\bb H\cong H\times H &&\bb T\cong\{(h,h)\ :\ h\in H\} 
&&\bb A\cong\{(h,h^{-1})\ :\ h\in H\},
 \end{align*}
so in particular the conjugation giving $G$ is 
$\sig(g_1,g_2)=(\ovl{g_2},\ovl{g_1}).$
So we are looking for pairs $(a,b)\in \bb G$ such that
$(a,b)\bb T(a,b)^{-1}$ is $\sig-$stable, modulo the action of $G.$ It
follows that for any $t_1=(h_1,h_1)\in \bb T$ there must be
$t_2=(h_2,h_2)\in \bb T$ such that
\begin{equation}
  \label{eq:stable}
  \sig\big( (a,b)t_1 (a,b)^{-1}\big)=(a,b)t_2(a,b)^{-1}.
\end{equation}
We conclude that $\Ad (\ovl{b}^{-1}a), \Ad(b^{-1}\ovl{a})\in N_G(H).$
In other words if we define $x:=\ovl{b}^{-1}a,$ then $x,\ovl{x}\in N_G(H).$ 

Multiplying any $(a,b)$ by $(\ovl{c},c)$ does not change anything;
$\Ad(a,b)\bb T$ is replaced by a $G-$conjugate. So we can take $b=1,$
and then $a$ will satisfy $a,\ovl{a}\in N_G(H).$ 

Write $a=ks$. Then both  $ks$ and $ks^{-1}$ must belong to $N_G(H),$
so it follows that $s^2\in N_G(H).$ Since $\Ad s$ is semisimple with
real positive eigenvalues, it follows that $s\in C_G(H)=H,$ and
therefore $k\in N_G(H)$.  So we can use $k$ instead of $a,$ and $k$
must satisfy $k^2\in H.$ Thus the conjugacy classes have representatives
\begin{equation}
  \label{eq:creps}
  \{(whw^{-1},h)\ :\ w\in N_K(H)/Z_K(H):=W, \text{ and } w^2\in H\}.
\end{equation}
Suppose that two of these, corresponding to $w_1,w_2\in W$  give
$G-$conjugate sets. Then there is $g\in G$ such that for any $h_1\in
H$ there is $h_2\in H$ such that
\begin{align*}
  &gw_1h_1w_1^{-1}g^{-1}=w_2h_2w_2^{-1}, && \ovl{g}h_1\ovl{g}^{-1}=h_2.
\end{align*}
It follows that $\ovl{g}\in N_G(H)$ and therefore also ${g}\in
N_G(H).$  As earlier, if $g=ks$ and $\sig (g)=ks^{-1}$ stabilize an
object, then so does $g^{-1}\sig(g)=s^{-2}$. Since $s$ is semisimple
and has real positive eigenvalues only, $s$ stabilizes the object as
well. Thus the sets 
$$
\{(sw_1hw_{1}^{-1}s^{-1},h)\}=\{(sw_1s^{-1}hsw_1^{-1}s^{-1},shs^{-1}=h)\}
$$ 
and $\{(w_1hw_{1}^{-1},h\}$ give the same $G-$conjugacy class. We
conclude that $k\in N_G(H)$ and $s\in H.$ So we may as well assume
$g=k\in N_K(H).$  But then  
\begin{equation}
  \label{eq:conclusion}
  \Ad(gw_1g^{-1})h=\Ad(w_2)h,
\end{equation}
for all $h\in H,$ so that $w_1$ and $w_2$ are conjugate as elements of
the Weyl group $W.$
We have proved the following result.
\begin{theorem}
\label{tw conj cx}
The conjugacy classes of twisted Cartan subgroups are in
one to one correspondence with conjugacy classes of involutions in the
Weyl group. More precisely, for any strongly regular semisimple $r\in \wti G,$ there is a unique
conjugacy class of an involution $w$ such that $r$ is conjugate by $G$
to an element of the form $(\theta w) h$ where $h\in H_w$ with 
\begin{equation*}
H_w:=\{ h\in H\ :\ \Ad(\theta w)(h)=h\}.
\end{equation*}
\end{theorem}
The claim follows from the following additional result; 
any element of the form $(\theta w)h$ with $h\in H$ is conjugate by
$H$ to  an element of the form  $(\theta w) h_w$ with $h_w\in H_w.$  
  Let $H^\pm_w:=\{ h\in H\ |\ \theta w(h)=h^{\pm 1}\}.$ 
In the complex case,  any element $h\in H_w^-$ can be written as $h=\theta w(r)r^{-1}$ for some $r\in H$,
 because  the map
$$
r\mapsto\theta w(r) r^{-1}
$$
is onto $H_w^-.$ Therefore any element of the form 
$\theta w th$ with $t\in H_w^+$ and $h\in H_w^-$ is conjugate to $\theta w t.$

\subsection{$GL(2n,\bR)$}\label{s:gln} We treat this case in detail, $GL(2n+1,\bb R)$ is similar.

The fundamental Cartan subalgebra can be realized as
$$
\fk h_0=\{\diag\big(
\begin{pmatrix} t_1&\theta_1\\-\theta_1&t_1\end{pmatrix}\dots 
\begin{pmatrix} t_n&\theta_n\\-\theta_n&t_n\end{pmatrix}\big)\}
$$
with the usual conjugation $\theta(x)=-x^t.$
Its complexification  can be written as 
\begin{align*}
&\fk h=\{ z:=(z_1,\dots ,z_n,z_{n+1},\dots ,z_{2n}),\ z_i\in\mathbb C\}\\
&\theta(z_1,\dots ,z_{2n})=-w_0(z)=-(z_{2n},\dots , z_1),\\
&\sigma(z)=w_0(\ovl{z})=(\ovl{z_{2n}},\dots ,\ovl{z_1})\\
&\fk t=\{(z_1,\dots ,z_n,-z_n,\dots ,-z_1)\}.
\end{align*}
The roots $\ep_i-\ep_{2n+1-i}$ are noncompact imaginary, and they form a maximal set of strongly orthogonal noncompact roots. The other conjugacy classes of Cartan subalgebras are obtained by applying a set of $n-k$  Cayley transforms about a subset of these roots. A  representative is 
 \begin{align*}
&\fk h^{n-k}=\{ z:=(z_1,\dots ,z_{k},z_{k+1},\dots ,z_{2n-k},z_{2n+1-k},\dots ,z_{2n})\},\\
&\theta^{n-k}(z_1,\dots ,z_{2n})=(-z_{2n},\dots ,-z_{2n+1-k},-z_{k+1},\dots,-z_{2n-k},-z_k,\dots ,-z_1),\\
&\sigma^{n-k}(z)=(\ovl{z_{2n}},\dots ,\ovl{z_{2n+1-k}},\ovl{z_{k+1}},\dots,\ovl{z_{2n-k}},\ovl{z_k},\dots ,\ovl{z_1}).
\end{align*}
The $\fk t^{n-k}$ obtained from $\fk t$ by applying the Cayley transform is 
\[
\fk t^{n-k}=\{(z_1,\dots ,z_n,-z_n,\dots ,-z_1)\}.
\]
We consider the case $k=n,$ the fundamental Cartan subalgebra. The stabilizer of $\fk t$ is $W^\theta.$ The real Weyl group is also $W(G,H)=W^\theta.$ This group is formed of changes $z_i\longleftrightarrow z_{2n+1-i}$ and permutations $w$ such that the sets $\{1,\dots ,n\}$ and $\{n+1,\dots ,2n\}$ are preserved, and if $w(i)=j,$ then $w(2n+1-i)=2n+1-j$; call this last subgroup $W^C$. It is the diagonal inside $W^c\times W^c$ where the first $W^c\cong S_n$ acts on the first $n$ coordinates and fixes the last $n,$  and the second $W^c$ fixes the first $n$ coordinates and acts on the last $n$ coordinates. Representatives of the double cosets $W^\theta\backslash W/W^\theta$ are given by $W^c:=(1,W^c)$  This can be seen as follows. Composing  $w$ by $r\in W^\theta$ on the left, we can insure that the result takes $(x_1,\dots x_{n},-x_n,\dots ,-x_1)$ to $(\pm x_1,\dots,\pm x_n ,\dots )$; if $\pm x_1$ are both beyond place $n$, move $x_1$ to the first $n$ coordinates by reflecting with the appropriate $(i,2n+1-i)$. Continue this way until the first $n$ coordinates are $\pm x_1,\dots ,\pm x_n.$ Then use $W^C$ to order the $\pm x_i$ in increasing order. If now
$-x_i$ occurs, then compose $w$ with the reflection $r\in W^\theta$ about $(i,2n+1-i)$ on the right to change it so that $x_i$ is in the $i$th coordinate. 

Write $w^c$ for $(1,w^c).$   We check that $\sig(w^c\fk t)=w^c\fk t$ precisely when $w^c$ is an involution. Suppose $w^c(-x_i)=-x_j,$ and $w^c(-x_j)=-x_k$. If $k\ne i,$ then the fixed points of $\sig$ on $w^c\fk t$ have strictly smaller dimension than $n.$ 
Finally two involutions give $W^\theta-$conjugate $w^c\fk t$ spaces precisely when the two involutions are conjugate by $W^c.$  This follows from observing that conjugating $(1,w^c)$  by a $(w,w)\in W^C\subset W^\theta$ amounts to the same as conjugating $w^c\in W^c$ by $w\in W^c.$ 

\bigskip
The cases $k<n$  are similar. The centralizer of $\fk t^{n-k}$ is as before. The real Weyl group $W(G,H_0^{n-k})$ is the same as $W^\theta$ on the coordinates $1,\dots k, 2n+1-k,\dots ,2n$ and the full $S_{2n-2k}$ on the middle coordinates. Thus the conjugacy classes of $w\fk t^{n-k}$ are the same as for the fundamental Cartan subgroup for $GL(2k,\bb R),$  on the coordinates $1,\dots ,k,2n+1-k,\dots ,2n,$ and trivial on the middle coordinates.

\section{Kernel Computations}
\label{sec:ker}

Let $\tau$ be an automorphism of $G=KS$ of finite order (commuting with $\theta$), 
and write $G^+=G\rtimes\{1,\tau\}$. Then 
$G^+$ admits a Cartan decomposition $G^+=K^+ S$.
Let $P^+=M^+N\subset G^+$ be a parabolic subgroup so that
$M^+=M\rtimes\{1,\tau\};$ in particular, $\tau(N)=N.$ 
Let $(\rho, V)$ be an admissible representation of $M^+,$ and
$(\pi,\Ind_{P^+}^{G^+}\rho)$ be the induced representation. The
representation space is 
\[
\{ f:G^+\longrightarrow V_\rho\ :\ f(nm^+ g)=\rho(m^+)f(g)\}
\] 
with action of $G^+$ by translation on the right. The space can be
identified with \newline
$
\{ f:K\longrightarrow V_\rho\}.
$ 
Let $F\in C_c^\infty (G^+)$. Then $\pi(F)$ is defined as
\[
\pi(F):=\int_{G^+}\pi(g^+)F(g^+)\;
d g^+. 
\] 
It is well known  that $\pi(F)$ is given by integration against a
kernel:   
\[
\begin{aligned}
 \pi(F)f  (x) &=\int_{G^+} F(g^+)[\pi(g^+)f](x)\;dg^+=
\int_{G^+} F(x^{-1}g^+)f(g^+)\;dg^+.
\end{aligned}
\]
Write $G=PK,$ and let $g^+=\tau pk_1,\ x=k_2$. For the case 
$F\in C_c^\infty(G\tau),$ we can rewrite $\pi(F)f(k_2)$ as
\[
\int_K\big[\int_P F(k_2^{-1}\tau pk_1)\rho(\tau p)\;dp\big]\  f(k_1)\;dk_1.
\]

So the kernel is 
$$
\int_PF(k_2^{-1}\tau pk_1)\rho(\tau p)\;dp.
$$ 
The distribution character is 
\[
\Theta_\pi(F)=\int_K\int_P F(k^{-1}\tau p k)\Tr\rho(\tau p)\; dp\; dk.
\]
As in the untwisted case, for fixed $\tau m\in M^+_{\reg},$ the map 
\[
\begin{aligned}
&\Psi:N\longrightarrow N\\
&n\longrightarrow\Ad(\tau m)(n) n^{-1}  
\end{aligned}
\]
is 1-1 and onto, so we can rewrite the integral as
\begin{equation}
  \label{eq:tchar}
\begin{aligned}
\Theta_\pi(F)=&\int_{M_{\reg}}\Delta(m)\Tr\rho(\tau m)\int_K\int_N F(k^{-1}n^{-1}\tau
  m n k)\;dn\;dk\ dm=\\
=&\int_{M_{\reg}} \Delta(m)\Tr\rho(\tau m) \int_{G(\tau m)\backslash G}
F(g\tau mg^{-1})\; dg\;{dm}, 
\end{aligned}  
\end{equation}
where $G(\tau m)$ is the centralizer of $\tau m$ in $G,$ and
$\Delta (m)$ is the appropriate Jacobian for the map $\Psi.$ 

Recall that there are finitely many twisted Cartan subgroups in $M,$
label them $\tau H_1,\dots ,\tau H_k.$ Then we can rewrite
(\ref{eq:tchar})  as
\begin{equation}
  \label{eq:tchar2}
  \Theta_\pi(F)=\sum_{i=1}^k \int_{H_i} D(h_i)\Tr\rho(\tau
  h_i)\int_{H_i\backslash G} F(g\tau h_i g^{-1})\; dg.
\end{equation}
$\Delta$ and $D$ are Jacobians, we need not make them explicit.
By the generalization of Bouaziz of the results of Harish-Chandra,
$\Tr \rho$ and $\Theta_\pi$ are given by integration against  an
analytic locally $L^1-$function on the regular set. 
\begin{corollary}\label{c:twistedzero}
 $\Theta_\pi$ is zero on any twisted Cartan subgroup which is not
 conjugate to one in $M^+.$ 
\end{corollary}
\begin{proof}
This follows from  formula (\ref{eq:tchar2}).
\end{proof}

\section{Indices of standard modules}
\label{sec:indices_std}

\subsection{Standard modules for complex groups} \ 
We are using standard notation for complex groups and their
  representations. See \eg \cite{BV} for a detailed treatment. Note
  however that there is a difference in that in the reference,
  $\la_L-\la_R$ is a weight of  the compact torus $T,$ and $\la_L
  +\la_R$ is a character of $A.$ 

Let $X(\la_L,\la_R)$ be a standard module with
Langlands quotient $\ovl{X}(\la_L,\la_R)$.  The Cartan subgroup is
$H=TA$ and the parameter corresponds to $\mu=\la_L +\la_R\in\wht T,$
$\nu=\la_L-\la_R\in\wht A.$ Since $\theta\mu=\mu$ and
$\theta\nu=-\nu,$ $\ovl{X}(\la_L,\la_R)$ extends to an irreducible
module of $G^+$ in two distinct ways if there is $w$ such that
$w\mu=\mu, w\nu=-\nu.$ If on the other hand there is no such $w\in W,$
then there is a unique irreducible module 
$\ovl{X}_{G^+}(\la_L,\la_R)$ which restricts to $G$ as $\ovl{X}(\la_L,\la_R)\oplus
\ovl{X}(\la_R,\la_L)$. 
The characters of these latter modules are 0 on $\wti G.$ So we only
consider the first kind.  In this case $\la_L$ is conjugate to
$\la_R$,  so we can write the parameter as
$(\la,w\la)$ for some $w\in W.$ Assume $2\la$ is regular; it is
already integral since it equals $\mu\in\wht T$ for the case
$\ovl{X}(\la,\la).$ 
Since there
must be $x\in W$ such that 
\[
\begin{aligned}
&x(\la +w\la)=\la+w\la,\\
&x(\la -w\la)=-\la+w\la,
\end{aligned}
\] 
it follows that $x=w$ is an involution. 

Assume that $\nu\ne 0,$ and let $P=MN$ be the parabolic subgroup such that 
\[
\begin{aligned}
&\Delta (M,H)=\{\al\ : (\al,\nu)=0\},\\
&\Delta (N,H)=\{\al\ : (\al,\nu)>0\}.  
\end{aligned}
\] 
Then $X(\la,w\la)=\Ind_P^G[X_M(\la,w\la)]$, and 
$X_M(\la,w\la)=X_M(\mu/2,\mu/2)\otimes\bb C_\nu$. The module
$X_M(\mu/2,\mu/2)$ is tempered, and it extends in two ways to an
irreducible module for $M^+.$  We will construct the two ways
below. Note that, being tempered,  
$X_M(\mu/2,\mu/2)=\ovl{X}_M(\mu/2,\mu/2)$. Denote by $(\rho,V_\rho)$
one such  $M^+-$module that has $\bX(\mu/2,\mu/2)$ as its Harish-Chandra
module. 

The element $\theta$ does not stabilize $N,$ so it is awkward to
define an action on $X(\la,w\la).$ However $\tau:=\theta w$ does
stabilize $N$, so it is natural to define its action as follows. 
Let $\rho(\theta):V_\rho\longrightarrow V_\rho$ be the intertwining operator 
satisfying
$\rho(\theta)\rho (m)=\rho(\theta(m))\rho(\theta).$ Since $w\mu=\mu$,
we can assume $w\in M,$ so $\rho(w)$ is well defined.   Denote by
$\pi$ the action of $G$ on $X(\la,w\la).$ Then for $f\in X(\la,w\la),$
define
\[
[\pi(\tau)f](x):=\rho(\tau^{-1})f(\tau (x)).
\]

We now define the action of $\theta$ in  the case
$\ovl{X_M}(\mu/2,\mu/2)=X_M(\mu/2,\mu/2)$. We suppress the subscript
  $M$ since this can be thought of as the case of $G$ and $\nu=0.$  
Let $(\fk b,\fk h)$ be a $\theta-$stable (complex) pair
of a Borel subalgebra and a Cartan subalgebra. The module $X(\la,\la)$
is tempered irreducible (therefore also unitary), and derived functor induced:
\[
\C L_{\fk b}^i(\bb C_{\tilde\mu})=
\begin{cases}
  X(\la,\la) &\text{ if } i=\dim \fk n\cap \fk k, \\
  0          &\text{ otherwise.}  
\end{cases}
\]
Here $\tilde\mu = 2\lambda-2\rho$ and we are using the unnormalized 
cohomological induction as in \cite{KV}, Chapter 5.
 
There are two ways to normalize the action of $\theta$ on
$X_\ep(\la,\la).$ The first one is to require that $\theta$ act by
$\ep$ on the lowest $K-$type $\mu=2\la.$ The second one is to    
denote by $\bb C_{\tilde\mu;\eta}$ the 
$T^+-$module which is equal to $\bbC_\wti\mu$ as a $T-$module and on 
which $\theta$ act by $\eta$. Then $X_\eta(\la,\la)$ is the 
$(\frg,K^+)-$module cohomologically induced from $\bb
C_{\tilde\mu;\eta}$ viewed as an $(\fk h,T^+)$ module.
We will use the first normalization, and the relation between the two
is   $\eta=\ep(-1)^{\dim (\fru\cap\frs)}$ (see the end of the proof of
  Theorem \ref{t:twcx}).

\begin{theorem}\label{t:twcx} Normalize the action of $\theta$ on
  $X_\ep(\la,\la)$ so that it acts by $\ep$ on the lowest $K-$type $\mu=2\la.$ 
  The twisted index of $X_\ep(\la,w\la)$ is 
\[
{I_\theta}[X_\ep(\la,w\la)]=
\begin{cases}
  0 &\text{ if } w\ne 1,\\
r\ep E_{2\lambda-\rho} &\text{ if } w=1,
\end{cases}
\]
where $r=[\Spin:E_\rho]$.
\end{theorem}

\begin{proof} Assume first that $w\ne 1.$ There are two
distinct actions of $\theta$ on $X_M(\mu,\mu)$, they lift to
$X(\la,w\la),$ and therefore also to 
$\ovl{X}(\la,w\la).$  The resulting modules are denoted
${X}_\ep(\la,w\la)$ and $\ovl{X}_\ep(\la,w\la)$ with $\ep=\pm 1.$
By Frobenius reciprocity, $\theta$ acts by the same scalar on the lowest
$K-$type of $X(\la,w\la)$ and the lowest $M\cap K-$type of
$X_M(\mu,\mu)$. This sign is not important for us since the index will
be shown to be 0 in both cases.

Let $P=MN$ be the parabolic subgroup determined by $\nu=\la-w\la.$
Since $\nu\ne 0,$ $P$ is a proper parabolic subgroup stabilized by
$w\theta.$ The kernel calculation in Section \ref{sec:ker} implies that
the distribution character is supported on $\Ad G(w\theta M)$ which does
not intersect $\theta T_0$ (Corolllary \ref{c:twistedzero}). Therefore the character is 0 on $\theta T_0.$ 
Formula  \ref{eq:twcharacter} implies that the index is 0 as claimed. 

\medskip
Assume now $w=1.$ 
By the usual arguments, one sees that Dirac cohomology of
$X=X(\lambda,\lambda)$ is obtained as the PRV component of the tensor product
of the lowest $K-$type of $X,$ and the spin module $S$.
Namely, any $K-$type of $X$ has highest weight of the form
$2\lambda+\sum_\beta n_\beta\beta$, 
where $\beta$ are positive roots and $n_\beta$ nonnegative integers.
On the other hand, any weight of $S$ is of the form
$-\rho+\sum_\beta m_\beta\beta$, with 
each $m_\beta$ being 0 or 1. Putting $k_\beta=n_\beta+m_\beta$, we see
that $H_D(X)$ consists of $K^\dagger-$modules $E_\tau$ satisfying 
\[
\tau+\rho=2\lambda+\sum_\beta k_\beta\beta
\]
with $\tau +\rho$ conjugate to $2\la,$ the infinitesimal character of $X$
restricted to $\fk t.$  In particular, $\|\tau+\rho\|^2=\|2\lambda\|^2$,
so 
\[
2\langle 2\lambda,\sum_\beta k_\beta\beta\rangle +\|\sum_\beta
k_\beta\beta\|^2 = 0. 
\]
Since each of the summands is nonnegative, they all have to be 0, so
all $k_\beta=0$, which implies the claim. Thus $H_D(X)$ 
is a single $K^\dagger-$ type $E_{2\lambda-\rho}$, with multiplicity $r=[\Spin:E_\rho]$.

It remains to consider the action of $\theta$ on the lowest $K-$type
of $X$, which is also the lowest $K^+-$type of
$X_\ep(\lambda,\lambda)$.  This lowest $K-$type $V$ is in the
bottom layer, as in section V.6 of \cite{KV}. Corollary 5.85 of \cite{KV} gives
\[
\Hom_{K^+}[\C L_S(\bbC_{\tilde\mu}),V]\cong\Hom_{T^+}[\bb
C_{\tilde\mu}\otimes\twedge^R(\fk n\cap\fk s),V^{\fk n\cap\fk k}]
\]
with $R=\dim(\fk n\cap \fk s)$ and $S=\dim(\fk n\cap\fk k).$
The action of $\theta$ on $V$ is by the same scalar
$\eta$ as the action on $V^{\frn\cap\frk}$. The action on $Z$ is by
$\ep$, and the action on $\twedge^R\frn\cap\frs$ is by
$(-1)^R=(-1)^{\dim\frn\cap\frs}$. 
\end{proof}

\subsection{Standard modules for real groups}
\label{std_real} 

We use the results, and some of the notation, in \cite{KV}, particularly chapter XI.  
The Langlands classification exhibits every irreducible module as a canonical quotient of a  standard module. The precise definition and statements of results are summarized in chapter 11 of \cite{ABV}. We will only use the following description of the standard modules.
\bed[Standard Module]
\label{d:langlands}
A standard module is a (Harish-Chandra  induced) module 
$$
X(P,\delta,\nu)=\Ind_{P}^G[\delta\otimes \bC_\nu]
$$ 
where the data $(P,\delta,\nu)$ are as follows.
\begin{enumerate}
\item $P=MAN$ is a real parabolic subgroup of $G$ with $M$ equal rank,
\item $\delta$ is a limit of discrete series of $M,$
\item $\nu\in \fk a^*$ satisfies $\langle \Rea\nu,\al\rangle \ge 0$, and satisfies the ``parity condition'' (11.10g) in \cite {ABV}.
\end{enumerate}
\ebed

\bethm[Langlands classification]
With $(P,\delta,\nu)$ as in Definition \ref{d:langlands}, $X(P,\delta,\nu)$ has a unique irreducible quotient denoted  $\bX(P,\delta,\nu).$ Two such modules are equivalent if and only if their parameters $(P,\delta,\nu)$ and $(P',\delta',\nu')$ are conjugate under $G.$
Any irreducible admissible module is equivalent to an $\bX(P,\delta,\nu).$
\ebethm
Recall  that $G$  is the real points of a linear algebraic reductive connected group. We follow \cite{KV} for the parametrization of the limits of discrete series. The group $M$ is equal rank but possibly disconnected. Let $\fk h=\fk t +\fk a$ be a $\theta-$stable Cartan subalgebra such that $\fk t$ is a compact Cartan subalgebra of $\fk m.$ Let $H=TA$ be the corresponding Cartan subgroup; by the assumptions on the group it is abelian. A datum for a limit of discrete series is a $\theta-$stable Borel subalgebra $\fk b_M$ containing $\fk t$ and an irreducible representation $\lambda$ of $T,$ such that 
\eq
\label{wgood_range}
\langle d\lambda+\rho(\fru),\alpha\rangle \geq 0,\qquad \alpha\in\Delta(\fru),
\eeq
where $S=\dim\fru\cap\frk$. See \cite{KV}, Chapter XI. 

\bed
Let $\fk b=\fk b_M+\fk n$ be the Borel subalgebra containing $\frh.$ We will denote by $A_\frb(\la)$ the derived module $\C L_S(\la)$ where we view $\la$ as a 1-dimensional $(\frh,T)-$module consisting of the datum for $\delta$ and $\nu$.
\ebed

The results in \cite{KV}, particularly  \cite{KV}, Theorem 11.129 (c), imply that also
$$
\C L_q(\la)=
\begin{cases}
X(P,\delta,\nu) &\text{ if } q=S,\\
0  &\text{ if } q\ne S.
\end{cases}
$$

\begin{theorem}
\label{index_std_real}
Assume $\frg$ and $\frk$ have equal rank.  The index of the standard module $A_\frb(\lambda)$ is 
\[
I(A_\frb(\lambda))=
\begin{cases}
  0 &\text{ if } \frb \text{ is not }\theta-\text{stable},\\
E_{\lambda+\rho(\fru\cap\frs)} &\text{ if } \frb \text{ is }\theta-\text{stable}.
\end{cases}
\]
\end{theorem}

\begin{proof}
If $\frb$ is $\theta-$stable, then it is well known and easy to see that 
$H_D(A_\frb(\lambda))$ is a single $K^\dagger-$type $E_{\lambda+\rho(\fru\cap\frs)}$. 
The computation is essentially the same as in the proof of Theorem \ref{t:twcx}. 
See also \cite{HP1}, \cite{HKP}. This $K^\dagger-$type appears in the tensor product 
of the lowest $K-$type of $A_\frb(\lambda)$ with the $K^\dagger-$type of the spin 
module $S$ containing the element 1. Hence 
$H_D(A_\frb,\lambda)=H_D^+(A_\frb,\lambda)$, and the result follows.

Assume now that $\frb$ is not $\theta-$stable. 
In view of (\ref{index formula}), the result will follow if we prove that the 
character of $A_\frb(\lambda)$ vanishes on the compact Cartan subalgebra. 
This can be proved by expressing the standard module $A_\frb(\lambda)$ 
using induction in stages, as real induced from a cohomologically induced 
module; see \cite{KV}, Theorem 11.172 and Corollary 11.173. To do this, 
we consider the group $MA=Z_G(\fra)$ with $M$ equal rank, 
and we let $P=MAN$ be the 
associated real parabolic subgroup of $G$. Then $T$ is a compact
Cartan subgroup  of $M$, and  the Borel subalgebra $\frb_\frm=\frb\cap\frm$
of $\frm$ is  $\theta-$stable. 
It follows that the standard module $A_\frb(\lambda)$ is 
Harish-Chandra  induced from an $A_{\frb\cap\frm}(\lambda_M)\otimes\bC_\nu$ of
$MA\subset P$  to $G$. Now we can apply 
the kernel computations of Section \ref{sec:ker} to conclude that the
character of $A_\frb(\lambda)$ is zero on the compact Cartan subgroup. 
\end{proof}

We now drop the equal rank assumption, and we consider modules for
the extended group $G^+$.
To understand what the standard modules are in this case, we first
note that twisting a standard $(\frg,K)-$module 
$A_\frb(\lambda)$ by $\theta$, we get 
$A_{\theta\frb}(\theta\lambda)$. 
The reason is as follows. The
$(\fk b,T)-$module $\bC_\la$ with the action twisted by $\theta$, 
is isomorphic to the $(\theta\fk b,T)$
module $\bC_{\theta\la},$ and $\bC_\la^\#$ with the action twisted by $\theta$ 
is isomorphic to $\bC_{\theta\la}^\#$. Thus the $(\fk g, T)-$module 
$M(\fk b,\la):=U(\fk g)\otimes_{U(\fk b)}\bC_\la$, with the action twisted by
$\theta$, is isomorphic to
$M(\theta\fk b,\theta\la):=U(\fk g)\otimes_{U(\theta\fk b)}\bC_{\theta\la}.$ 
 The action of $\fk g$ on the corresponding derived
module comes from the action of $\fk g$ on  $\Hom[R(K),M(\fk
b,\la)]$ given by $(X\cdot F)(k)=(\Ad k (X))F((k)).$ Twisting the action
by $\theta$ yields the action of $\fk g$ on $\Hom[R(K),M(\theta\fk
b,\theta\la)]$. 

This leads to three cases:
\begin{enumerate}
\item $\theta\frb=\frb$ and $\theta\lambda=\lambda$;
\item $A_{\theta\frb}(\theta\lambda)$ is not isomorphic to $A_{\frb}(\lambda)$;
\item $A_{\theta\frb}(\theta\lambda)$ is isomorphic to
  $A_{\frb}(\lambda)$, but $\theta\frb\neq\frb$. In this case, we may assume that $\theta\lambda\neq\lambda$; otherwise, we could modify $\frb$ to be in Case 1.
\end{enumerate}

\noindent{\bf Case 1.} Now $A_\frb(\lambda)$ is a module for $G^+$,
in two ways, distinguished by the sign $\ep=\pm 1$ by which $\theta$
acts on the unique lowest $K-$type. We denote the module corresponding
to $\ep$ by $A_\frb^\ep(\lambda)$.  

Another way to obtain these modules
is to specify an action of $\theta$ on the $(\frh,T)-$module $\bbC_\lambda$
to make it into an $(\frh,T^+)-$module, and then use cohomological
induction to obtain a $(\frg,K^+)-$module. To link the two constructions,
we have to compare the actions of $\theta$ on $\bbC_\lambda$ and on
the lowest $K-$type of $A_\frb(\lambda)$. Assume that $\theta$ acts by $\ep'=\pm 1$ 
on $\bC_\la.$ Since
$\bC_\la^\#=\bigwedge^{\top}\fru\otimes \bC_\la,$ $\theta$ acts on
$\bC_\la^\#$ by $\ep'(-1)^R$ where $R=\dim\fru\cap\frs$. 
Let $V$ be a $K^+-$type with $\theta$ acting by $\eta=\pm 1$. 
By Proposition 5.71 in \cite{KV},
\[
\Hom_{K^+}[\C L_j^{K^+}(\bC_\la),V]\cong\Hom_{T^+}[\bC_\la^\#,H^j(\ovl{\fru}\cap \fk k,V)].
\]
Setting $j=S=\dim(\fru\cap\fk k),$ we find that this is nonzero
precisely when $\eta=(-1)^R\ep'.$ On the other hand, by the results on
page 365 and Theorem 5.80 of \cite{KV}, there is a one-to-one $K^+-$equivariant
\textbf{bottom layer map}
$$
\C B:\C L^{K^+}_S(\bC_\la)\longrightarrow \C L_S(\bC_\la)=A_{\fk b}(\la).
$$
Its image is called the \textbf{bottom layer $K^+-$types}.
By  Proposition 10.24 and Chapter V.6 in \cite{KV}, 
the lowest $K^+-$type of $A_{\fk b}^\ep(\la)$ is the unique $K^+-$type in
the bottom layer $K^+-$ types. So we see that the module $A_{\fk
  b}(\la)$ with $\theta$ action by $\ep$ on
the lowest $K-$type is obtained by applying cohomological induction to
the module $\bC_\la$ with $\theta-$action $\ep(-1)^R$.

As in the proof of Theorem \ref{index_std_real}, the Dirac cohomology
of $A_\frb(\lambda)$ is equal to the $K^\dagger-$type
$E_{\lambda+\rho(\fru\cap\frs)}$, 
but the multiplicity is
no longer equal to 1. Rather, the multiplicity is equal to
the multiplicity
of irreducible summands of the spin module, which is equal to $[\frac{1}{2}\dim\fra]$ (recall that $\frh=\frt\oplus\fra$). This is the same as $\dim S_{(\frh,T)}$, where $S_{(\frh,T)}=\bigwedge\fra^+$ is the spin module for the pair $(\frh,T)$. The only $K-$type of $A_\frb(\lambda)$
  contributing to the Dirac cohomology is the lowest $K-$type, so the
  action of $\theta$ on the Dirac cohomology is by $\ep$. So
\[
I_\theta(A_\frb^\ep(\lambda))=\ep E_{\lambda+\rho(\fru\cap\frs)}\otimes S_{(\frh,T)}.
\] 

\noindent{\bf Case 2.}  
The corresponding standard module for $G^+$ is {$M_{\frb,\lambda}=A_\frb(\lambda)\oplus
A_{\theta\frb}(\theta\lambda)$, with the two summands interchanged 
by $\theta$. More precisely, fixing an isomorphism $\Phi$ from the $\theta-$twisted $A_\frb(\lambda)$
onto $A_{\theta\frb}(\theta\lambda)$, we can define the action of $\theta$ on $v+w\in A_\frb(\lambda)\oplus
A_{\theta\frb}(\theta\lambda)$ by
\[
\theta(v+w)=\Phi^{-1}w + \Phi v \quad \in A_\frb(\lambda)\oplus
A_{\theta\frb}(\theta\lambda)
\]
Choosing a different isomorphism $\Phi$ would lead to an isomorphic module.

Now it follows that 
$I_\theta(M_{\frb,\lambda})=0$, because $H_D(M_{\frb,\lambda})$
decomposes into two pieces interchanged by $\theta$, and hence the trace of $\theta$ on this space is zero.

\medskip
\noindent{\bf Case 3.}  In this case there is $k\in K$ such that
  $\Ad(k)\theta\frb=\frb$, and also $\Ad(k)\theta\lambda=\lambda$. 
We  can write $\frb=\frh\oplus\frn$ with $\theta\frh=\frh$ and
  $\Ad(k)\frh=\frh$. Let $\chi=\lambda+\rho_\frb$ denote the
  infinitesimal character of $A_\frb(\lambda)$. We assume $\chi$ is
  regular; for singular $\chi$ the result then follows by translation
  principle. Since $\theta\frh=\frh$, we can decompose
  $\frh=\frt\oplus\fra$ into eigenspaces for $\theta$. Let
  $\chi=(\sigma,\nu)$ be the corresponding decomposition of $\chi$.  

The above $k$ determines an element $w$ of the Weyl group $W(G,H)$, and
\[
w(\sigma,\nu)=(\sigma,-\nu).
\]
Moreover, $k$ can be chosen so that $k^2=(k\theta)^2$ is in the center
of $G$; so the above $w$ is an involution. 

Similarly to Case 1, we have two options for an action of $k\theta$, which can
be distinguished by the sign they produce on the lowest $K-$type. We denote
the corresponding standard modules by $A_\frb^\ep(\lambda)$, $\ep=\pm 1$.

We can compute the twisted character $\Theta$ of $A_\frb^\ep(\lambda)$
by using the kernel computation of Section
\ref{sec:ker} based on $\tau=k\theta$. 
Let $\frp=\frm\oplus\fru\supseteq\frb$ be the parabolic
subalgebra of $\frg$ 
corresponding to $\nu$; so $\frm$ is the centralizer of $\nu$. Then
$\tau\frm=\frm$ and $\tau\fru=\fru$. Let $H_0=T_0A_0$ be the fundamental Cartan subgroup,
and let $T_0'\theta$ be the set of strongly regular elements. 
By Corollary \ref{c:twistedzero}, 
$\Theta$ is supported on $\Ad(G)(Mk\theta)$. But 
$T'_0\theta$ does not meet $\Ad(G)(Mk\theta)$, so it follows
that $\Theta=0$ on $T_0'\theta$. By equation (\ref{eq:twcharacter}) in
Section \ref{sec:characters}, this implies that the twisted Dirac
index of $A_\frb^\ep(\lambda)$ is 0. 

We have proved:

\begin{theorem}
\label{tw_index_std_real}
If $\theta\frb=\frb$ and $\theta\lambda=\lambda$, the twisted Dirac indices of the standard $(\frg,K^+)-$modules
$A_\frb^\ep(\lambda)$ are
\[
I_\theta(A_\frb^\ep(\lambda))=\ep E_{\lambda+\rho(\fru\cap\frs)}\otimes S_{(\frh,T)}.
\] 
The twisted Dirac indices of all other standard modules are zero. \qed
\end{theorem}

\subsection{Applications}\label{sec:examples}
Consider the two special cases of Dirac index we have considered in the introduction. One is the ordinary Dirac index,
$I(X)=X\otimes (S^+-S^-)$ for a $(\frg,K)-$module $X$, in the equal rank case, while the other is the twisted Dirac index 
$I_\theta(X)=(X^+-X^-)\otimes S$ for a $(\frg,K^+)-$module $X$. In each case, $X$ is a finite length $(\fk g, K)-$module, 
and there is an involution $\gamma$ acting on $X\otimes S$, equal either to $1\otimes\theta$ or to $\theta\otimes 1$, and we have
\begin{equation}
  \label{eq:twcoh}
H_D^\pm(X)= r\sum m_\tau^\pm E_\tau 
\end{equation}
\begin{equation}
  \label{eq:index}
  I_\gamma(X)=r\big(\sum m_\tau^+ E_\tau-\sum m^-_\tau E_\tau\big)
\end{equation}
where $m^\pm_\tau$ are the multiplicities of $E_\tau$ in the $\pm 1
-$eigenspaces of $\gamma,$ and $r>0$ is an integer depending on 
$(\fk g, K)$ only. ($r$ is the multiplicity of the irreducible summands in the spin module.)
 On the other hand, 
 \begin{equation}
   \label{eq:character1}
 X=\sum m(X,\chi)X(\chi)  
 \end{equation}
where $X(\chi)$ are standard modules. Then
\begin{equation}
  \label{eq:indexchar}
  I_\gamma(X)=\sum m(X,\chi)I_\gamma(X(\chi)).
\end{equation}

Specialize $X$ to the cases of irreducible unipotent representations
in \cite{BP1}. Write $\bX:=\bX(\la,-x\la)$ for a unipotent
representation. One can verify in all cases that $\bX\big|_K$ decomposes with
multiplicity 1, and that $E_{2\la-\rho}$ occurs as the PRV component of
$E_\mu\otimes E_\rho.$ Thus one of $H_D^\pm(\bX)$ is zero, and $I_\gamma(\bX)=me(\bX)E_{2\la-\rho}$ with $e(\bX)=\pm
1.$  Combining this with  Theorem \ref{t:twcx} and Equation (\ref{eq:indexchar}), we conclude 
$$
e(\bX)=\eta m(\bX,(\la,-\la)),
$$ 
with $\eta$ as in Theorem \ref{t:twcx}. This can be used to solve for $m(\bX,(\la,-\la))$. 

\medskip
Similarly, specialize to the cases of irreducible unipotent
representations in \cite{BP2}. One can verify in all cases that one of
$m^\pm_\tau = 0$ (see below). Thus we can write \newline $I_\gamma(\bX)=\sum
e_\bX(\tau)m(\tau)E_\tau.$ Theorem \ref{tw_index_std_real} implies
$$
e_\bX(\tau_\chi)m(\tau_\chi)=m(\bX,\chi).
$$

In all cases,  (\ref{eq:char}) or (\ref{eq:twcharacter}) give explicit formulas for
the distribution characters restricted to the elliptic strongly
regular set, $\theta T'_0$ where $T_0$ is the $\theta-$fixed part of
the fundamental Cartan subgroup $H_0=T_0A_0.$ 

We now explain why in the examples of unipotent representations  in \cite{BP2}, one of $m^\pm_\tau$ has to be 0 for each $\tau$. 
Let $(\frg,K)$ be the pair corresponding to $G=Sp(2n,\bbR)$ or $G=U(p,q)$. Let 
$\frt\subset\frk\subset\frg$ be a compact Cartan subalgebra, and choose compatible
positive root systems $R^+(\frg,\frt)\supset R^+(\frk,\frt)$. Denote $W=W(\frg,\frt)$, $W_\frk=W(\frk,\frt)$ and 
let $W^1$ be the subset of $W$ consisting of $w$ which take the fundamental $(\frg,\frt)-$chamber
into the fundamental $(\frk,\frt)-$chamber.

Recall that for each representation $X$ that we studied, we considered the candidates for $K^\dagger-$types that can appear in $H_D(X)$:
\[
\tau=w\Lambda-\rho,
\]
where $\Lambda\in\frt^*$ is the infinitesimal character of $X$ and $w\in W^1$. The multiplicity of each such $\tau$ was calculated in \cite{BP2}; it is equal to the number of solutions to the equation
\[
w_1\tau= \sigma\rho-\rho_\frk +\mu^-,
\]
with variables $w_1\in W_\frk$, $\sigma\in W^1$, and $\mu^-$ the lowest weight of a $K-$type of $X$.

To compute the index $I(X)$, we must determine the multiplicity of
each $\tau$ in $H_D^+(X)$ and in $H_D^-(X)$. We claim that for each
fixed $\tau$ one of these multiplicities is 0. Equivalently, we claim
that all $\sigma$ corresponding to a fixed $\tau$ have the same
sign. We explain why this is so for   
$G=Sp(2n,\bbR)$; the reasoning for $U(p,q)$ is similar.

We identify $\sigma$ with $\sigma\rho$, and recall that each
$\sigma\rho$ corresponding to our fixed $\tau$ has a fixed ``core"
part plus possibly one or two more fixed coordinates. The variable
part of $\sigma\rho$ was 
composed of pairs of neighboring integers that could either be placed
on the left of the core, as $(i+1,i)$, or on the right of the core, as
$(-i,-i-1)$. Moreover, the number of negative pairs, which lie to the
right of the core, is fixed, \ie does not depend on $\sigma$ (or even
on $\tau$). We denote this number by $r$. (The number of positive
pairs is thus also fixed.) 

It is now easy to see that all $\sigma$ corresponding to the same
$\tau$ have equal parity; one can be obtained from another by
exchanging some pairs, 
and moving each pair requires an odd number of steps. One can moreover
compute that the sign attached to $\tau$ is $(-1)^r$ times the sign of
the fixed part of $\sigma$. 

In particular, we see that $I(X)$, and so also the character of $X$ on
the compact Cartan subgroup $T$ of $G$, can not be zero whenever $X$
has nonzero Dirac cohomology. So each such $X$ is elliptic, and using
the formulas for $H_D(X)$ in \cite{BP2} together with the above
remarks, we can calculate the character on $T$ explicitly. 

\section{Extensions of modules}
\label{s:ext}

Recall from the introduction that the Euler-Poincar\'e pairing of two (virtual) $(\frg,K)-$modules $X$ and $Y$ is the virtual
vector space 
\[
\EP(X,Y)=\sum_{i=1}^s (-1)^i\Ext^i_{(\frg,K)}(X,Y).
\]

As explained in the introduction, we should understand $\EP(X,Y)$ in case $X$ is a standard
module $A_\frb(\lambda)$ of Section \ref{std_real}. 
For this we use a special case of the ``Frobenius reciprocity spectral sequence" of \cite{greenbook}, Proposition 6.3.2 and Corollary 6.3.4.
(For the $\theta-$stable case, see also \cite{KV}, Corollary 5.121,
or \cite{ic2}, Theorem 6.14, where this spectral sequence is called
the Zuckerman spectral sequence.) 

\begin{proposition}
\label{zuckspseq} 
There is a first quadrant spectral sequence with differential of bidegree $(r,1-r)$, and $E_2$ term
\eq
\label{zuck_e2}
E_2^{pq}=\Ext^p_{(\frh,T)}(\bbC_\lambda^\#,H^q(\bar\fru,Y)),
\eeq
converging to 
\eq
\label{zuck_limit}
\Ext^{p+q-S}_{(\frg,K)}(A_\frb(\lambda),Y).
\eeq
\end{proposition}

Since $H^q(\bar\fru,Y)$ are finite-dimensional $(\frh,T)-$modules by \cite{HS}, Theorem 7.22, it makes sense to take the alternating sum of spaces in (\ref{zuck_e2}). Since (\ref{zuck_limit}) is obtained from (\ref{zuck_e2}) by successively taking cohomology with respect to finitely many differentials, we can use the Euler-Poincar\'e principle to obtain
\[
\sum_{p,q} (-1)^{p+q}\Ext^p_{(\frh,T)}(\bbC_\lambda^\#,H^q(\bar\fru,Y))= (-1)^S \sum_i (-1)^i\Ext^i_{(\frg,K)}(A_\frb(\lambda),Y).
\]
In other words,
\eq
\label{EP_spseq}
\EP(A_\frb(\lambda),Y)=(-1)^S\sum_{p,q} (-1)^{p+q}\Ext^p_{(\frh,T)}(\bbC_\lambda^\#,H^q(\bar\fru,Y)).
\eeq

\subsection{EP pairing and Dirac index in the equal rank case} 

The main result of this subsection is

\begin{theorem} 
\label{thm_altsum}
Let $X$ and $Y$ be two finite length $(\frg,K)-$modules with infinitesimal character. Then
\eq
\label{EP_eqrk}
\EP(X,Y) = \Hom_{K^\dagger}(I(X),I(Y))
\eeq
in the Grothendieck group of finite-dimensional vector spaces.
\end{theorem}
\pf
As we already remarked, $\EP(X,Y)=0$ unless $X$ and $Y$ have the same infinitesimal character, 
and by Vogan's conjecture the same is true for the right side. Therefore we may assume that $X$ and $Y$ have the same infinitesimal character.
Furthermore, it is enough to prove the theorem in case $X$ is a standard module, since standard modules generate the Grothendieck group of finite length $(\frg,K)-$modules. 

So let $X=A_\frb(\lambda)$ as in Section \ref{std_real}. We first consider the case when $\frb$ is not $\theta-$stable, \ie  $\frh$ is not contained in $\frk$, \ie  
$\fra\neq 0$. We claim that in this case both sides of (\ref{EP_eqrk}) are equal to zero.

In view of (\ref{EP_spseq}), $\EP(A_\frb(\lambda),Y)=0$ follows from the following lemma, \ie  its obvious extension to virtual $(\frh,T)-$modules.
\begin{lemma}
\label{htext}
Let $V$ and $W$ be finite-dimensional $(\frh,T)-$modules, where $\frh=\frt\oplus\fra$ is a $\theta-$stable Cartan subalgebra of $\frg$ such that $a=\dim\fra > 0$. Then
\[
\sum_{p=0}^a (-1)^p\Ext^p_{(\frh,T)}(V,W)=0.
\]
\end{lemma}
\begin{proof} The space $\Ext^p_{(\frh,T)}(V,W)$ is the $p$th cohomology of the complex
\[
\Hom_T(\twedge^\cdot\fra\otimes V,W)=\twedge^\cdot\fra\otimes \Hom_T(V,W),
\]
with the usual de Rham differential. Here we used the fact that the adjoint action of $T$ on $\fra$ is trivial, so $\twedge^\cdot\fra$ can be pulled out of the Hom space. Since the complex is finite-dimensional, we can use the Euler-Poincar\'e principle and conclude that
\[
\sum_p (-1)^p\Ext^p_{(\frh,T)}(V,W)=(\twedge^\text{even}\fra-\twedge^\text{odd}\fra)\otimes \Hom_T(V,W).
\]
The last expression is however zero since $\twedge^\text{even}\fra$ and $\twedge^\text{odd}\fra$ have the same dimension.
\end{proof}

The right side of (\ref{EP_eqrk}) is zero as well, by Theorem \ref{index_std_real}. 

This leaves us with the case of $\theta-$stable $\frb$.
In this case, since $M(\frh,T)=M(T)$ is a semisimple category, the higher Ext groups in this category vanish, and so (\ref{EP_spseq}) becomes
\begin{multline}
\label{alt_ext_coho}
\EP(A_\frb(\lambda),Y)=(-1)^S\sum_q (-1)^q\Hom_T(\bbC_\lambda^\#,H^q(\bar\fru,Y))=\\
(-1)^S\Hom_T(\bbC_\lambda^\#,\sum_q (-1)^q H^q(\bar\fru,Y));
\end{multline}
the alternating sum can go inside $\Hom$, because all $H^q(\bar\fru,Y)$ are finite-dimensional. Note that the above arguments actually imply that the spectral sequence given by (\ref{zuck_e2}) and (\ref{zuck_limit}) collapses, and hence we in fact get
\eq
\label{explicit ext}
\Ext^{q-S}_{(\frg,K)}(A_\frb(\lambda),Y)= \Hom_T(\bbC_\lambda^\#,H^q(\bar\fru,Y))
\eeq
for every $q$. We are however interested primarily in the alternating sum, and not in the individual $\Ext$ groups.

To compare the right sides of (\ref{alt_ext_coho}) and (\ref{EP_eqrk}), we need the following two lemmas.
We denote by $D(\frk,\frt)$ the cubic Dirac operator for the pair $(\frk,\frt)$ \cite{G}, \cite{Ko2}, and by $I_\frk(E)$ the corresponding Dirac index of a $\frk-$module 
$E$. The spin module we take to define $I_\frk(E)$ is $S_\frk=\twedge(\fru\cap\frk)$. As usual, the difference between the adjoint and the spin action of $\frt$ on $\twedge(\fru\cap\frk)$ is given by a shift by $\rho(\fru\cap\frk)$. 

\begin{lemma}
\label{alt_coho}
For any admissible $(\frg,K)-$module $Y$ there is an equality of virtual $T-$modules
\[
\sum_q (-1)^q H^q(\bar\fru,Y)=I_\frk(I(Y))\otimes\bbC_{\rho(\fru)}.
\]
\end{lemma}
\begin{proof}
Writing out the right side we get 
\begin{multline*}
I_\frk(I(Y))\otimes\bbC_{\rho(\fru)}=I(Y)\otimes (S_\frk^+-S_\frk^-)\otimes\bbC_{\rho(\fru)}=\\ 
Y\otimes (S^+-S^-)\otimes (S_\frk^+-S_\frk^-)\otimes\bbC_{\rho(\fru)}=\\
Y\otimes (\twedge^{\text{even}}\fru\cap\frs-\twedge^{\text{odd}}\fru\cap\frs)\otimes (\twedge^{\text{even}}\fru\cap\frk-\twedge^{\text{odd}}\fru\cap\frk)=\\
Y\otimes (\twedge^{\text{even}}\fru-\twedge^{\text{odd}}\fru).
\end{multline*}
So our lemma is a special case of \cite{ic2}, Theorem 8.1. It is also a special case of \cite{HS}, Theorem 7.22 (which is stated for homology and not cohomology, but it is easy to pass between homology and cohomology using Poincar\'e duality.) 
\end{proof}

\begin{lemma}
\label{homK_homT}
Let $E_\mu$ be the irreducible finite-dimensional $\frk-$module with highest weight $\mu\in\frt^*$. Let $E$ be any finite-dimensional $\frk-$module. Then
\eq
\label{k-mult}
\Hom_\frk(E_\mu,E)=\Hom_\frt(\bbC_{\mu+\rho(\fru\cap\frk)},(-1)^S I_\frk(E)).
\eeq
\end{lemma}
\pf
Let $\nu$ be a dominant $\frk-$weight.  
By \cite{Ko2}, Dirac cohomology of $E_\nu$ with respect to $D(\frk,\frt)$ is the $\frt-$module
\[
H_D(E_\nu)=\bigoplus_{w\in W_\frk} \bbC_{w(\nu+\rho(\fru\cap\frk))}.
\]
It follows that 
\[
I_\frk(E_\nu)=\bigoplus_{w\in W_\frk} (-1)^{S+l(w)}\bbC_{w(\nu+\rho(\fru\cap\frk))};
\]
to check the signs, note that the highest Cartan component of the tensor product $E_\nu\otimes S_\frk$ corresponds to $w=1$ and to $\twedge^{\text{top}}\fru\cap\frk$.

The only dominant weight in the above expression for $I_\frk(E_\nu)$ corresponds to $w=1$, and is equal to $(-1)^S \bbC_{\mu+\rho(\fru\cap\frk)}$. 
Since $E=\bigoplus_\nu m_\nu E_\nu$, we can express the multiplicity $m_\mu$ as the multiplicity of the $\frt-$module 
$(-1)^S \bbC_{\mu+\rho(\fru\cap\frk)}$  in $I_\frk(E)=\bigoplus_\nu m_\nu I_\frk(E_\nu)$. This implies the lemma.
\epf

\begin{remark}
\label{singular}
Lemma \ref{homK_homT} includes the case when $\mu+\rho(\fru\cap\frk)$ is singular, in which case $E_\mu=0$.
\end{remark}

We now combine (\ref{alt_ext_coho}) with Lemma \ref{alt_coho}, remembering that $\bbC_{\lambda}^{\#}=\bbC_{\lambda+2\rho(\fru)}$. It follows
\[
\EP(A_\frb(\lambda),Y)=\Hom_{T^\dagger}(\bbC_{\lambda+\rho(\fru)},I_\frk(I(Y)),
\]
where $T^\dagger$ is the spin double cover of $T$. 
Using Lemma \ref{homK_homT}, we see that this is further equal to $\Hom_{K^\dagger}(E_{\lambda+\rho(\fru\cap\frs)},I(Y))$. The claim now follows from
Theorem \ref{index_std_real}. This finishes the proof of Theorem \ref{thm_altsum}.
\epf

\subsection{Elliptic pairing, equal rank case}
{
As mentioned in the introduction, Theorem \ref{thm_altsum}
is also related to the elliptic pairing of (finite length) $(\frg,K)-$modules. The elliptic pairing was defined by Arthur \cite{A} (see also \cite{H} and \cite{R}). 
{
We assume for simplicity that $G$ and $K$ are connected, and that $G$ contains a Cartan subgroup equal to a 
compact maximal torus $T\subseteq K$ (consisting of elliptic elements). 
}
Let $X,Y$ be irreducible 
$(\frg,K)-$modules, and let $\Theta_X$, $\Theta_Y$ denote
the characters of the corresponding group representations, viewed as functions on $G$. The elliptic pairing of $X$ and $Y$ is
\eq
\label{def ell}
\langle X,Y\rangle_{\ell} =\frac{1}{|W(G,T)|}\int_T |\det[(1-\Ad(t))\big|_{\frg/\frt}]| \Theta_X(t)\overline{\Theta_Y(t)} dt,
\eeq
where $W(G,T)$ denotes the Weyl group of $G$ with respect to $T$ and $dt$ denotes the normalized Haar measure on $T$.

The elliptic pairing is extended linearly to the Grothendieck group of finite length $(\frg,K)-$modules; as in Definition \ref{def virtual}, we call the elements of this Grothendieck group virtual $(\frg,K)-$modules.
By \cite{H} and \cite{R}, the elliptic pairing of two virtual $(\frg,K)-$modules is equal to their ``Dirac index pairing", \ie, to
\[
\langle X,Y\rangle_{\DI}=\int_{K^\dagger}\ch(I(X))(k)\overline{\ch(I(Y))(k)}dk=\Hom_{K^\dagger}(I(X),I(Y)).
\]
(Here we identified the virtual vector space $\Hom_{K^\dagger}(I(X),I(Y))$ with its dimension; see Definition \ref{def virtual}. The second equality then follows from the standard orthogonality relations for $K^\dagger$.)

So we conclude
\begin{corollary}
\label{cor:ell} 
For any two finite length $(\frg,K)-$modules $X$ and $Y$,
\[
\langle X,Y\rangle_{\ell}=\EP(X,Y)= \Hom_{K^\dagger}(I(X),I(Y)).
\]
\end{corollary} 

We note that the equality 
\eq
\label{ell DI}
\langle X,Y\rangle_{\ell}=\langle X,Y\rangle_{\DI}
\eeq 
can also be derived from \eqref{eq:character}. Namely, by Weyl integral formula we can write
\[
\langle X,Y\rangle_{\DI}=\frac{1}{|W(K,T)|}\int_{T^\dagger} \det[(1-\Ad(t))\big|_{\frk/\frt}]\ch(I(X))(t)\overline{\ch(I(Y))(t)}dt.
\]
On the other hand we can substitute \eqref{eq:character} into \eqref{def ell}. Since $W(G,T)=W(K,T)$, and since  
$\frg/\frt=\frk/\frt\oplus \frs$ implies
\[
\det[(1-\Ad(t))\big|_{\frg/\frt}]=\det[(1-\Ad(t))\big|_{\frk/\frt}]\det[(1-\Ad(t))\big|_{\frs}],
\] 
the equality \eqref{ell DI} will follow once we check
\[
\det[(1-\Ad(t))\big|_{\frs}]=|\ch(S^+-S^-)(t)|^2.
\]
This follows from 
\[
\twedge^{\text{even}}\frs-\twedge^{\text{odd}}\frs=(S^+-S^-)\otimes(S^+-S^-)^*,
\]
and from the following simple lemma applied to $A=-\Ad(t)$.

\begin{lemma}
\label{det tr}
Let $V$ be an $n-$dimensional complex vector space and let $A:V\to V$ be a diagonalizable linear operator. Let $\twedge^i(A):\twedge^i(V)\to\twedge^i(V)$ be defined by
\[
\twedge^i(A)(v_1\wedge\dots\wedge v_i)=Av_1\wedge\dots\wedge Av_i.
\]
Then 
\[
\det(1+A)=\sum_{i=0}^n \tr(\twedge^i(A)).
\]
\end{lemma}
\pf If $\lambda_1,\dots,\lambda_n$ are the eigenvalues of $A$, then the equality becomes
\[
\prod_{j=1}^n(1+\lambda_j)=\sum_{i=0}^n\ \sum_{1\leq j_1<\dots<j_i\leq n}\lambda_{j_1}\dots\lambda_{j_i},
\]
which is obviously true.
\epf
}

\subsection{The general case} \label{gencase}

We now drop the equal rank assumption. Let $X$ and $Y$ be modules for
the extended group $G^+=G\rtimes \{1,\theta\}$, and consider the twisted
Euler-Poincar\'e pairing $\EP_\theta(X,Y)$ defined in the introduction.

The main result of this subsection is 

\begin{theorem} 
\label{thm_altsum_theta}
Let $X$ and $Y$ be two finite length modules for $G^+$ (\ie
$(\frg,K^+)-$modules), with infinitesimal character. Then 
\eq
\label{EP_gen}
\EP_\theta(X,Y) = c\Hom_{K^\dagger}(I_\theta(X),I_\theta(Y))
\eeq
in the Grothendieck group of finite-dimensional vector spaces. The constant $c$ is as in the discussion following
  definition \ref{def:c}, particularly formula (\ref{eq:twistedext}).
\end{theorem}

The idea is as before, to prove the theorem in case $X$ is a 
standard module, when we can express $\EP_\theta(X,Y)$ by
(\ref{EP_spseq}).

We treat separately each of the three cases from Section \ref{std_real}. 

\noindent{\bf Case 1.}  Now $\theta\frb=\frb$ and $\theta\lambda=\lambda$,
and we have two standard $(\frg,K^+)$ modules $A_\frb^\ep(\lambda)$,
$\ep=\pm 1$, with $\theta-$action given by $\ep$ on the lowest $K-$type.

Since $\fru$ and $\bar\fru$ are $\theta-$stable, $\theta$ acts on the
complex $\Hom_\bbC(\twedge\bar\fru,Y)=\twedge\fru\otimes Y$ used to
compute the $\bar\fru-$cohomology of $Y$, by acting on both
$\twedge\fru$ and $Y$. This action commutes with the differential and
therefore descends to cohomology. 
Thus the right side of the identity (\ref{EP_spseq}) has a natural action of $\theta$.
As defined at the beginning of Section \ref{gencase}, so does the left side, and we claim that the two actions are compatible,
\ie  (\ref{EP_spseq}) is an equality of virtual $\{1,\theta\}-$modules. 
To see this, we invoke the Fundamental Spectral Sequence in
Section 8 Chapter V of \cite{KV}. Let $X,Y$ be $(\fk g,\wh
K)-$modules. The action of $\theta$ on $\Ext_{(\fk g, K)}^i(X,Y)$ is defined via
the complex $\Hom_K[\bigwedge^i\frs\otimes X ,Y].$ Standard
cohomological properties of $\Ext$ imply that we get the same action
on $\Ext_{(\fk g, K)}^i(X,Y)$ by using a complex $\Hom_{(\fk g, K)}[P^i,Y]$ with
$P^i\to X\to 0$ any projective $(\fk g,\wh
K)-$resolution. Using this definition of the action of $\theta$ we can
trace the action of $\theta$ through the spectral sequence in
Proposition 5.113 and the particular  cases in Theorem 5.120 of
\cite{KV}. The key fact needed is that 
the isomorphism in Proposition 2.34 of \cite{KV}, in our
notation 
$$
\Hom_{\fk g, K}[P_{\bar{\fk b},T}^{\fk g,K}(\bC_\la^\#),Y]\cong 
\Hom_{\bar{\fk  b}, T}[\bC_\la^\#,\mathcal{F}^\vee Y],
$$
is compatible with the $\theta-$action. This is straightforward.  

So taking the trace of $\theta$ on both sides of  (\ref{EP_spseq}), we get
\eq
\label{twisted_EP_spseq}
\Tr[\theta,\EP_\theta(A_\frb^\ep(\lambda),Y)]=(-1)^S\sum_p (-1)^p
\Tr[\theta,\Ext^p_{(\frh,T)}(\bbC_\lambda^\#,\sum_q (-1)^q
H^q(\bar\fru,Y))]. 
\eeq
Denoting the finite-dimensional virtual $(\frh,T)-$module $\sum_q
(-1)^q H^q(\bar\fru,Y)$ by $Z$, we can use the $(\frh,T)-$version of
(\ref{eq:twistedext}) and write the right side of
(\ref{twisted_EP_spseq}) as 
\eq
\label{twisted_EP_spseq2}
(-1)^Sc\Hom_{T^\dagger}(\Tr[\theta,\bbC_\lambda^\#]\otimes
S_{(\frh,T)},\Tr[\theta,Z]\otimes S_{(\frh,T)}). 
\eeq
Here $S_{(\frh,T)}$ is a spin module for the pair $(\frh,T)$,
constructed as $\twedge\fra^+$ for some maximal isotropic subspace
$\fra^+$ of $\fra$. 

Since $\theta$ acts  by $\ep$ on $\bbC_\lambda^\#$,  
$\Tr[\theta,\bbC_\lambda^\#]=\ep\bbC_\lambda^\#$   as a virtual 
$T^+-$module.

\medskip
For the other term in (\ref{twisted_EP_spseq2}), we use the following
analogue of Lemma \ref{alt_coho}. 

\begin{lemma}
\label{alt_coho_unequal}
Let $\frb=\frh\oplus\fru$ be a $\theta-$stable Borel subalgebra of
$\frg$. Let $Y$ be an admissible module for $G^+$.  
Then
\[
\Tr[\theta,\sum_q (-1)^q H^q(\bar\fru,Y)]\otimes 
S_{(\frh,T)}=I_\frk(I_\theta(Y))\otimes\bbC_{\rho(\fru)}
\]
as virtual vector spaces. 
\end{lemma}
\begin{proof}
As in the proof of Lemma \ref{alt_coho}, we can use \cite{ic2},
Theorem 8.1, or \cite{HS}, Theorem 7.22, to identify  
\[
\sum_q (-1)^q H^q(\bar\fru,Y)= 
Y\otimes (\twedge^{\text{even}}\fru-\twedge^{\text{odd}}\fru)
\]
as virtual $(\frh,T)-$modules. 
This identification is compatible with $\theta-$actions. 

On the other hand, 
\begin{multline*}
\Tr[\theta,Y\otimes
(\twedge^{\text{even}}\fru-\twedge^{\text{odd}}\fru)]\otimes
S_{(\frh,T)}=\\ 
\Tr[\theta,Y\otimes(\twedge^{\text{even}}\fru\cap\frs-
\twedge^{\text{odd}}\fru\cap\frs)\otimes 
(\twedge^{\text{even}}\fru\cap\frk-\twedge^{\text{odd}}\fru\cap\frk)]\otimes
S_{(\frh,T)}=\\ 
(Y^+-Y^-)\otimes\twedge\fru\cap\frs\otimes S_{(\frh,T)}\otimes
(\twedge^{\text{even}}\fru\cap\frk-\twedge^{\text{odd}}\fru\cap\frk). 
\end{multline*}
Since the spin module $S$ for $(\frg,K)$ can be identified with
$\twedge\fru\cap\frs\otimes S_{(\frh,T)}\otimes
\bbC_{-\rho(\fru\cap\frs)}$, and the spin 
module $S_\frk$ for $(\frk,T)$ can be identified with
$\twedge\fru\cap\frk\otimes\bbC_{-\rho(\fru\cap\frk)}$, the above
expression is equal to 
\[
(Y^+-Y^-)\otimes S \otimes (S_\frk^+-S_\frk^-)\otimes\bbC_{\rho(\fru)}
= I_\frk(I_\theta(Y))\otimes\bbC_{\rho(\fru)}. 
\] 
\end{proof}
Since $\bbC_\lambda^\#=\bbC_{\lambda+2\rho(\fru)}$, Lemma
\ref{alt_coho_unequal} implies that (\ref{twisted_EP_spseq2}) is equal
to 
\eq
\label{twisted_EP_spseq3}
(-1)^Sc\ep\Hom_{T^\dagger}(\bbC_{\lambda+\rho(\fru)}\otimes
S_{(\frh,T)},I_\frk(I_\theta(Y))). 
\eeq
By Lemma \ref{homK_homT}, this is further equal to 
\eq
\label{twisted_EP_spseq4}
c\ep\Hom_{K^\dagger}(E_{\lambda+\rho(\fru\cap\frs)}\otimes S_{(\frh,T)},I_\theta(Y)).
\eeq
On the other hand, by Theorem \ref{tw_index_std_real}, $I_\theta(A_\frb^\ep(\lambda))=\ep
E_{\lambda+\rho(\fru\cap\frs)}\otimes S_{(\frh,T)}$.  So (\ref{twisted_EP_spseq})-(\ref{twisted_EP_spseq4}) imply
\eq
\label{twisted_EP_spseq5}
\Tr[\theta,\EP_\theta(A_\frb^\ep(\lambda),Y)]=
c\Hom_{K^\dagger}(I_\theta(A_\frb^\ep(\lambda)),I_\theta(Y)),
\eeq
and this proves Theorem \ref{thm_altsum_theta} in this case.

\medskip
\noindent{\bf Case 2.} For the standard module $M_{\frb,\lambda}$, 
$\EP_\theta(M_{\frb,\lambda},Y)$ decomposes into
two pieces interchanged by $\theta$, so 
\[
\Tr[\theta,\EP_\theta(M_{\frb,\lambda},Y)]=0
\]
as a virtual vector space. (Note that it is {\it not} equal to 0 as a
virtual $\{1,\theta\}-$module.) Since also 
$I_\theta(M_{\frb,\lambda})=0$ by Theorem \ref{tw_index_std_real}, 
 the statement of  Theorem \ref{thm_altsum_theta} 
holds trivially in the present case.  
 
\medskip
\noindent{\bf Case 3.} Recall that there is $k\in K$ such that $k\theta\frb=\frb$ and 
$k\theta\lambda=\lambda$. By Theorem \ref{tw_index_std_real} 
 the right side of the equality claimed by Theorem
\ref{thm_altsum_theta} is $0$  for $X=A_\frb(\lambda)$. 
To see that the
left side of that equality is also $0$, we first note that $K$ acts trivially
on the complex $\Hom_K(\bigwedge^\cdot\frs\otimes X,Y)$ which computes the
$\Ext^i_{(\frg,K)}(X,Y)$. 
Hence $K$ also acts trivially on $\EP(X,Y)$, and so the action of
$\theta$ on $\EP(X,Y)$ is the same as the action of $k\theta$. 
On the other hand, similarly as in Case 1, $k\theta$ acts on both sides of 
(\ref{EP_spseq}), in a compatible way, and thus it is enough to show that the trace
of $k\theta$ on the right side of (\ref{EP_spseq}) is 0. For this, we need an analogue of 
Lemma \ref{htext}. Note that $\frh=\frt\oplus\fra$
with $\fra\neq 0$. Since $\theta\lambda\neq\lambda$, $\lambda$ is not identically 0 on $\fra$.
So $k\theta\lambda=\lambda$ implies that $k\theta$ has a nonzero $+1-$eigenspace in $\fra$. 
Denoting this eigenspace by $\fra_1$, we can write $\fra=\fra_1\oplus\fra_2$
with $\fra_2$ invariant for $k\theta$.

\medskip
Let $\tilde\tau$ be the automorphism of the pair $(\frh,T)$ acting by $k\theta$. Let $T^\diamond=T\rtimes\{1,\tilde\tau\}$.
Let $V$ and $W$ be virtual $(\frh,T^\diamond)-$modules. Like in the proof of Lemma \ref{htext}, we can write
\begin{multline*}
\sum_p (-1)^p\Ext^p_{(\frh,T)}(V,W)=
(\twedge^\text{even}\fra-\twedge^\text{odd}\fra)\otimes \Hom_T(V,W)=\\
(\twedge^\text{even}\fra_1-\twedge^\text{odd}\fra_1)
\otimes (\twedge^\text{even}\fra_2-\twedge^\text{odd}\fra_2)\otimes\Hom_T(V,W).
\end{multline*}
Thus $\Tr[\tilde\tau, \sum_p (-1)^p\Ext^p_{(\frh,T)}(V,W)]$ is 0 as 
a virtual vector space, since the action of $\tilde\tau$ on 
$(\twedge^\text{even}\fra_1-\twedge^\text{odd}\fra_1)$ is trivial, 
so the trace of $\tilde\tau$ on it is equal to the dimension, which is 0.

\bigskip
This completes the proof of Theorem \ref{thm_altsum_theta}.

\subsection{Elliptic Pairing, Twisted Case}

{
As in the untwisted case, the two sides of the equality in Theorem \ref{thm_altsum_theta} are equal to the {\it twisted elliptic pairing} of $X$ and $Y$, defined by
\eq
\label{def twell}
\langle X,Y\rangle_{\ell,\theta} =\frac{1}{|W|}\int_{T^+} |\det[(1-\Ad(t^+))\big|_{\frg/\frt}]| \Theta_X(t^+)\overline{\Theta_Y(t^+)} dt^+,
\eeq
where as before $W=W(G,T)=W(K,T)$ is the Weyl group, and $dt^+$ denotes the normalized Haar measure on $T^+$.
Here $X$ and $Y$ are finite length $(\frg,K^+)-$modules, or virtual $(\frg,K^+)-$modules. 
The definition of $\langle\,,\,\rangle_{\ell,\theta}$ is included in the considerations of Arthur \cite{A}; it can also be found in \cite{W2}, Section 7.3. 
{
We assume for simplicity that $G$ and $K$ are connected, and that $H=TA$ is a Cartan subgroup of $G$ such that $T\subseteq K$ is a maximal torus (consisting of elliptic elements).  The case of most interest is when  $H\neq T$, \ie, $G$ and $K$ do not have equal rank. So we will assume this.
}

Since $T^+=T\cup \wti T=T\cup \theta T$, the integral in \eqref{def twell} is the sum of two pieces, and we claim that the integral over $T$ is in fact 0. This will follow if we prove
\[
\det[(1-\Ad(t))\big|_{\frg/\frt}]=0,\qquad t\in T.
\]
Since $\frg/\frt=\frk/\frt\oplus\frs$ as $T-$modules, it is enough to prove 
\eq
\label{T triv}
\det[(1-\Ad(t))\big|_{\frs}]=0,\qquad t\in T.
\eeq
To see this, we use Lemma \ref{det tr} to rewrite the left hand side as 
\[
\tr[\Ad(t);\twedge^{\text{even}}\frs-\twedge^{\text{odd}}\frs].
\]
We can decompose $\frs=\fra\oplus\frs'$ as $T-$modules; recall that we are assuming $\fra\neq 0$. Then
\[
\twedge^{\text{even}}\frs-\twedge^{\text{odd}}\frs=(\twedge^{\text{even}}\fra-\twedge^{\text{odd}}\fra)\otimes (\twedge^{\text{even}}\frs'-\twedge^{\text{odd}}\frs'),
\]
and since $\Ad(t)$ is the identity on $\fra$, its trace on $\twedge^{\text{even}}\fra-\twedge^{\text{odd}}\fra$ is 0. This proves \eqref{T triv}, and we see that the integration in \eqref{def twell} is only over $\wti T=\theta T$, \ie, 
\eq
\label{twell T+}
\langle X,Y\rangle_{\ell,\theta} =\frac{1}{|W|}\int_{T} |\det[(1-\Ad(\theta t))\big|_{\frg/\frt}]| \Theta_X(\theta t)\overline{\Theta_Y(\theta t)} dt.
\eeq
On the other hand, we define the twisted Dirac index pairing of $X$ and $Y$ as 

\begin{multline*}
\langle X,Y\rangle_{\DI,\theta}=\Hom_{K^\dagger}(I_\theta(X),I_\theta(Y))=\int_{K^\dagger}\ch(I_\theta(X))(k)\overline{\ch(I_\theta(Y))(k)}\,dk=\\
\frac{1}{|W|}\int_{T^\dagger} |\det[(1-\Ad(t))\big|_{\frk/\frt}]|\ch(I_\theta(X))(t)\overline{\ch(I_\theta(Y))(t)}\,dt.
\end{multline*}
(The last equality uses the Weyl integral formula.)

We claim that the elliptic pairing of $X$ and $Y$ is equal to $c$ times their twisted Dirac index pairing, with $c$ as in Theorem \ref{thm_altsum_theta}. To see this, we follow the discussion below Corollary \ref{cor:ell}: we 
compare the above formula for $\langle X,Y\rangle_{\DI,\theta}$ with the expression for $\langle X,Y\rangle_{\ell,\theta}$ obtained by substituting  
\eqref{eq:twcharacter} into \eqref{twell T+}, \ie, with
\[
\langle X,Y\rangle_{\ell,\theta} =\frac{1}{|W|}\int_{T} |\det[(1-\Ad(\theta t))\big|_{\frg/\frt}]| \frac{\ch(I_\theta(X))(t)\overline{\ch(I_\theta(Y))(t)}}{|\ch(S)(t)|^2} dt.
\]
Since $\theta$ is $1$ on $\frk/\frt$ and $-1$ on $\frs$, writing $\frg/\frt=\frk/\frt\oplus\frs$ implies
\[
\det[(1-\Ad(\theta t))\big|_{\frg/\frt}]=\det[(1-\Ad(t))\big|_{\frk/\frt}]\det[(1+\Ad(t))\big|_{\frs}].
\]
By Lemma \ref{det tr}, 
\[
\det[(1+\Ad(t))\big|_{\frs}=\ch(\twedge(\frs))(t)=c\ch(S\otimes S^*)(t)=c|\ch(S)(t)|^2.
\]
So we see that indeed $\langle X,Y\rangle_{\ell,\theta}=c\langle X,Y\rangle_{\DI,\theta}$. 

The corollary summarizes the discussion.

\begin{corollary}
\label{cor:twell} 
For any two finite length $(\frg,K^+)-$modules $X$ and $Y$,
\[
\langle X,Y\rangle_{\ell,\theta}=\EP_\theta(X,Y)= c\Hom_{K^\dagger}(I_\theta(X),I_\theta(Y)).
\]
\end{corollary} 
}


\begin{thebibliography}{ALTV}

\bibitem [AB]{AB} J.~Adams, D.~Barbasch,
\emph{Reductive dual pair correspondence for complex groups},
J. Funct. Anal. \textbf{132} (1995), no. 1, 1--42.

\bibitem[ABV]{ABV} 
J.~Adams, D.~Barbasch, D.~Vogan Jr., {\em The
Langlands Classification and Irreducible Characters for Real Reductive
Groups}, Birkh\"auser, Boston-Basel-Berlin, 1992.

\bibitem[ALTV]{altv}
J.~Adams, M.~vanLeuwen, P.~Trapa, D.~Vogan,
{\em  Unitary representations of real rediuctive groups}, preprint,
arXiv:1212.2192v2

\bibitem [A]{A} J.~Arthur, \emph{On elliptic tempered characters}, Acta Math. \textbf{171} (1993), 73--138.

\bibitem [B]{B} D.~Barbasch,
\emph{The unitary dual for complex classical Lie groups}, Invent. Math. \textbf{96} (1989), no. 1, 103--176.

\bibitem[Bs]{Bs} D.~Barbasch, \emph{The unitary spherical spectrum  for split classical  groups},
J. Inst. Math. Jussieu \textbf{9} (2010), no. 2, 265--356

\bibitem[BCT]{BCT}
D.~Barbasch, D.~Ciubotaru, P.~Trapa, \emph{Dirac cohomology for graded affine Hecke algebras},
Acta Math. \textbf{209} (2012), no. 2, 197--227.

\bibitem [BP1]{BP1} D.~Barbasch, P.~Pand\v zi\'c,
\emph{Dirac cohomology and unipotent representations of complex groups}, in \emph{Noncommutative Geometry and Global Analysis}, A.~Connes, A.~Gorokhovsky, M.~Lesch, M.~Pflaum, B.~Rangipour (eds), Contemporary Mathematics vol. 546, American Mathematical Society, 2011, pp. 1--22.

\bibitem [BP2]{BP2} D.~Barbasch, P.~Pand\v zi\'c, 
\emph{Dirac cohomology of unipotent representations of $Sp(2n,\bb R)$ and $U(p,q)$}, J. Lie Theory \textbf{25} (2015), no. 1, 185--213.

\bibitem [BV]{BV} D.~Barbasch, D.~Vogan,
\emph{Unipotent representations of complex semisimple groups},
Ann. of Math. \textbf{121} (1985), 41--110.

\bibitem [BW]{BW} A.~Borel, N.R.~Wallach,
\emph{Continuous cohomology, discrete subgroups, and representations of reductive groups}, second edition,
Mathematical Surveys and Monographs 67, American Mathematical Society, Providence, RI, 2000.

\bibitem [Bz]{Bz} 
A.~Bouaziz, \emph{Sur les charact\`eres des groupes de Lie
r\'eductifs non  connexes,} J. Funct. Anal. \textbf{70} (1987), 1--79.

\bibitem[CH]{CH} D.~Ciubotaru, X.~He, \emph{Green polynomials of Weyl groups, elliptic pairings, and the extended Dirac index}, Adv. Math. \textbf{283} (2015), 1--50.

\bibitem[CT]{CT} D.~Ciubotaru, P.~Trapa, \emph{Characters of Springer representations on elliptic conjugacy classes}, Duke Math. J. \textbf{162} (2013), no. 2, 201--223.

\bibitem[DH]{DH} C.-P.~Dong, J.-S.~Huang, \emph{Dirac cohomology of cohomologically induced modules for reductive Lie groups}, Amer. J. Math. \textbf{137} (2015), no. 1, 37--60.

\bibitem[G]{G} S.~Goette, \emph{Equivariant $\eta-$invariants on homogeneous spaces}, Math. Z. \textbf{232} (1998), 1--42.

\bibitem[HC]{HC} Harish-Chandra, {\it The characters of semisimple Lie groups}, 
Trans. Amer. Math. Soc. \textbf{83} (1956), 98--163. 

\bibitem[HS]{HS} H.~Hecht, W.~Schmid, \emph{Characters, asymptotics and $\frn-$homology of Harish-Chandra modules} 
Acta Math. \textbf{151} (1983), 49--151.

\bibitem [H]{H} J.-S.~Huang, \emph{Dirac cohomology, elliptic representations and endoscopy}, in \emph{Representations of reductive groups}, M.~Nevins, P.~Trapa (eds), Progr. Math. 312, Birkh\"auser/Springer, Cham, 2015, pp. 241--276.

\bibitem [HKP]{HKP} J.-S.~Huang, Y.-F.~Kang, P.~Pand\v zi\'c, \emph{Dirac
cohomology of some Harish-Chandra modules}, Transform. Groups \textbf{14} (2009), no. 1, 163--173.

\bibitem [HMS]{HMS} J.-S.~Huang, D.~Mili\v ci\'c, B.~Sun, \emph{Kazhdan's orthogonality conjecture for real reductive groups}, arXiv:1509.01755.

\bibitem [HP1]{HP1} J.-S.~Huang, P.~Pand\v zi\'c, \emph{Dirac
cohomology, unitary representations and a proof of a conjecture of
Vogan}, J. Amer. Math. Soc.  \textbf{15} (2002), 185--202.

\bibitem [HP2]{HP2} J.-S.~Huang, P.~Pand\v zi\'c, \emph{Dirac
Operators in Representation Theory}, Mathematics: Theory and
Applications, Birkh\"auser, 2006.

\bibitem[HPR]{HPR} J.-S.~Huang, P.~Pand{\v{z}}i{\'c}, D.~Renard, 
\emph{Dirac operators and Lie algebra cohomology}, Represent.
Theory \textbf{10} (2006), 299--313.

\bibitem [Kn]{Kn} A.W.~Knapp, \emph{Representation Theory of Semisimple Groups: An Overview Based on Examples},
Princeton University Press, Princeton, 1986. Reprinted: 2001.

\bibitem [KV]{KV} A.W.~Knapp, D.A.~Vogan, Jr., \emph{Cohomological induction and unitary representations},
Princeton University Press, Princeton, 1995.

\bibitem [Ko1]{Ko} B.~Kostant,
\emph{Lie algebra cohomology and the generalized Borel-Weil
theorem}, Ann. of Math. \textbf{74} (1961), 329--387.

\bibitem [Ko2]{Ko2} B.~Kostant, \emph{A cubic Dirac operator and the emergence of 
Euler number multiplets of representations for equal rank subgroups},
Duke Math. J. \textbf{100} (1999), 447--501.

\bibitem [MPV]{MPV} S.~Mehdi, P.~Pand\v zi\'c, D.A.~Vogan, Jr., \emph{Translation principle for Dirac index}, to appear in Amer. J. Math.

\bibitem [PS]{PS} P.~Pand\v zi\'c, P.~Somberg, \emph{Higher Dirac cohomology of modules with generalized infinitesimal character}, Transform. Groups \textbf{21} (2016), no. 3, 803--819.

\bibitem [P1]{P1} R.~Parthasarathy, \emph{Dirac operator and the discrete
series}, Ann. of Math. \textbf{96} (1972), 1--30.

\bibitem [P2]{P2} R.~Parthasarathy, \emph{Criteria for the unitarizability of some highest weight modules},
Proc. Indian Acad. Sci. \textbf{89} (1980), 1--24.

\bibitem [PRV]{PRV} K.R.~Parthasarathy, R.~Ranga Rao, V.S.~Varadarajan, \emph{Representations
of complex semisimple Lie groups and Lie algebras}, Ann. of Math. \textbf{85} (1967), 383--429.

\bibitem [R]{R} D.~Renard, \emph{Euler-Poincar\'e pairing, Dirac index and elliptic pairing for Harish-Chandra modules}, J. \'Ec. Polytech. Math. \textbf{3} (2016), 209--229.

\bibitem[St]{St} R.~Steinberg, \emph{Endomorphisms of Linear Algebraic Groups}. Mem. Amer. Math. Soc. \textbf{80} (1968). 

\bibitem [V1]{ic2} D.A.~Vogan, Jr., \emph{Irreducible characters of semisimple Lie groups II. The Kazhdan-Lusztig conjectures},
Duke Math. J. \textbf{46} (1979), no. 4, 805--859.

\bibitem [V2]{greenbook} D.A.~Vogan, Jr., \emph{Representations of real reductive
    Lie groups}, Birkh\"auser, Boston, 1981. 

\bibitem [V3]{V} D.A.~Vogan, Jr., \emph{Dirac operators and unitary
representations}, 3 talks at MIT Lie groups seminar, Fall 1997.

\bibitem [V4]{ic4} D.A.~Vogan, Jr., \emph{Irreducible characters of semisimple Lie groups IV. The Kazhdan-Lusztig conjectures},
Duke Math. J. \textbf{49} (1982), no. 4, 943--1073.

\bibitem [W1]{W1} J.-L.~Waldspurger, \emph{Les facteurs de transfert pour les groupes classiques: une formulaire},
Manuscripta Math. \textbf{133} (2010), no. 1-2, 41--82.

\bibitem [W2]{W2} J.-L.~Waldspurger, \emph{La formule des traces locale tordue},
arXiv:1205.1100.

\bibitem [Zh]{Zh} D.P.~Zhelobenko,
\emph{Harmonic analysis on complex semisimple Lie groups}, Mir, Moscow, 1974.

\end{thebibliography}
\end{document}